\documentclass{amsart}

\usepackage{amssymb}
\usepackage[mathscr]{euscript}

\usepackage{cite} 
\usepackage{enumerate}

\newtheorem{thm}{Theorem}[section]
\newtheorem{prop}[thm]{Proposition}
\newtheorem{lem}[thm]{Lemma}
\newtheorem{cor}[thm]{Corollary}

\theoremstyle{definition}
\newtheorem{definition}[thm]{Definition}
\newtheorem{example}[thm]{Example}
\newtheorem*{acknowledgement}{Acknowledgements}

\theoremstyle{remark}

\newtheorem{remark}[thm]{Remark}

\numberwithin{equation}{section}

\newcommand{\tl}[1]{\tilde{#1}}
\newcommand{\ol}[1]{\overline{#1}}
\newcommand{\cL}{\mathcal{L}}
\newcommand{\cC}{\mathcal{C}}
\newcommand{\comment}[1]{}

\begin{document}


\title{Steady Ricci Solitons on Complex Line Bundles}


\author{Maxwell Stolarski}





\begin{abstract}
We show the existence and uniqueness of a one-parameter family of smooth complete $U(1)$-invariant gradient steady Ricci solitons on the total space of any complex line bundle over a Fano K\"ahler-Einstein base with first Chern class proportional to that of the base.
These solitons are non-K\"ahler except on the total space of the canonical bundle.
\end{abstract}


 \maketitle


\section{Introduction}
A \textit{Ricci soliton} $(M, g, V, \epsilon)$ is a manifold $M$ together with a Riemannian metric $g$, a vector field $V$ on $M$, and a real constant $\epsilon$ such that
\begin{equation} \label{generalsolitoneqn}
Rc(g) = \mathscr{L}_V g + \epsilon g,
\end{equation}
where $\mathscr{L}_V g$ denotes the Lie derivative of $g$ with respect to $V$.
A Ricci soliton is called \textit{expanding, steady,} or \textit{shrinking} if $\epsilon$ is negative, zero, or positive respectively.
A Ricci soliton $(M, g, V, \epsilon)$ is called \textit{gradient} if the vector field $V$ is the gradient of a function on $M$.
Notice that if $V$ is a Killing vector field then the metric is Einstein, so Ricci solitons can be regarded as generalizations of Einstein metrics.
Ricci solitons are fixed points of the Ricci flow in the space of metrics modulo scaling and diffeomorphism, and they often arise as singularity models of the Ricci flow.

Riemannian manifolds of cohomogeneity one, that is, those where a group acts isometrically with a generic orbit of codimension one, give a natural class of examples on which to investigate the existence of Ricci solitons.
In this case, the soliton equation (\ref{generalsolitoneqn}) reduces to a system of ordinary differential equations in a variable transverse to the orbits.
Many of the known examples are of this form.
For example, Cao ~\cite{Cao96} demonstrated the existence of $U(n)$-invariant gradient steady Ricci soliton metrics on the total space of the canonical line bundle over complex projective space $\mathbb{C}P^{n-1}$. 
These solitons were later generalized by Feldman, Ilmanen, and Knopf in ~\cite{FIK03} who found smooth shrinking and expanding Ricci solitons on the total space of certain line bundles over complex projective space.
These examples are also K\"ahler and were unified and generalized in ~\cite{DancerWang08}.
Specifically, in ~\cite{DancerWang08}, Dancer and Wang developed the general framework for cohomogeneity one Ricci solitons, observed that the same equations arise in more general settings, and produced examples of steady, expanding, and shrinking K\"ahler-Ricci solitons on the total space of certain complex vector bundles over a product of Fano K\"ahler-Einstein manifolds (which are not necessarily homogeneous) that satisfy a particular relationship between first Chern classes.
Outside of the K\"ahler setting, Ivey ~\cite{Ivey94} produced a one-parameter family of doubly-warped gradient steady solitons on the total space of any trivial real vector bundle over an Einstein base with positive scalar curvature.
Buzano, Dancer, and Wang later found additional examples of such non-K\"ahler steady Ricci soliton metrics on this space when the base is a product of Einstein manifolds with positive scalar curvature in ~\cite{BDW15} and ~\cite{DancerWang09}.

As mentioned above, the cohomogeneity one Ricci soliton equations arise in more general settings when the metrics on the hypersurfaces depend on a single transverse variable.
This paper considers the case where the hypersurfaces are principal $U(1)$-bundles equipped with metrics such that the bundle projections are Riemannian submersions to a Fano K\"ahler-Einstein base.
Unlike the cohomogeneity one case, the hypersurfaces here need not be homogeneous.
While Dancer and Wang ~\cite{DancerWang08} obtained K\"ahler examples and Ivey ~\cite{Ivey94} obtained examples on trivial bundles, the general case is still not well understood as remarked in ~\cite{BDW15}.
Bergery ~\cite{BB82} showed the existence of Einstein metrics on such spaces.
Our result gives existence and uniqueness of non-Einstein steady Ricci soliton metrics on manifolds of the type considered by Bergery, namely
\begin{thm} \label{mainthm}
If $E$ is the total space of a complex line bundle $E \rightarrow B$ over a Fano K\"ahler-Einstein base $B$ such that the first Chern class $c_1(E)$ of $E$ is a rational 
multiple of $c_1(B)$ in $H^2(B, \mathbb{R})$, then there is a one-parameter family of non-homothetic smooth complete $U(1)$-invariant gradient steady Ricci solitons on $E$.
\comment{
Moreover, every smooth complete $U(1)$-invariant gradient steady Ricci soliton $(E, G, \nabla h)$, where $G$ is of the form
$$G = ds^2 + f(s)^2 g_{2 \pi} + g(s)^2 \check{g} \qquad \text{ on } E \setminus B_0 \approx (0, \infty) \times P$$
and $h = h(s)$ is a function of $s \in (0, \infty)$, arises in this one-parameter family and is uniquely classified by its maximum scalar curvature and volume of the zero section $B_0 \subset E$.}
\end{thm}

\comment{ 
Insert remark about why the gradient condition can be dropped in the uniqueness statement
$V$ is of the form
$$V = \xi (s) \partial_s + \lambda(s) \hat{U} \qquad \text{ on } E \setminus B_0 \approx (0, \infty) \times P$$
}

\noindent Note that the Chern class relationship is a purely topological condition.
If additionally $c_1(B)$ spans $H^2(B, \mathbb{R})$, as is the case for $B = \mathbb{C}P^n$, then the total space of \textit{any} complex line bundle over $B$ admits a one-parameter family of complete steady Ricci soliton metrics.
The corresponding uniqueness statement will be stated in section \ref{uniquenessSection}.

The paper is organized as follows:
First, we describe a collection of $U(1)$-invariant metrics on $E$ such that a gradient steady Ricci soliton metric in this collection corresponds to a solution of a boundary value problem for a nonlinear system of ordinary differential equations.
Next, a topological fixed point argument is applied to deduce the existence and uniqueness of solutions to this boundary value problem and hence of the corresponding Ricci soliton metrics.
Finally, we show that a nonempty subset of these Ricci soliton metrics are complete and have nonnegative Ricci curvature.
In the final section, we also show that if $c_1(E) \ne -c_1(B)$ then every soliton in this one-parameter family is necessarily non-K\"ahler with respect to the natural complex structure.
If $c_1(E) = -c_1(B)$ then every soliton in this one-parameter family is K\"ahler, and we obtain an alternate construction of gradient steady K\"ahler-Ricci solitons constructed by Dancer and Wang in ~\cite{DancerWang08}.

The original version of this paper appeared in November of 2015.
In July of 2017, Matthias Wink pointed out a gap in the original proof of completeness.
We are grateful for that observation and correct this gap in this version.
Since then, related results by Wink ~\cite{Wink17} and Appleton ~\cite{Appleton17} have appeared.

\begin{acknowledgement}
The author would like to thank Dan Knopf for his mentorship and initially suggesting the problem.
The author was partially supported by NSF RTG grant DMS-1148490.
\end{acknowledgement}

\section{Setup}
Let $(B^{d}, \check{g}, \check{J}, \check{\omega})$ be a smooth K\"ahler-Einstein manifold of real dimension $d$ with positive Einstein constant normalized such that $\check{Rc} = (d+2) \check{g}$, and
let $E \rightarrow B$ be a complex line bundle over $B$ with first Chern class $c_1(E)$ a multiple of $c_1(B)$ in $H^2(B, \mathbb{R})$, say $c_1(E) = q c_1(B)\in H^2(B, \mathbb{R})$ for some $q \in \mathbb{R}$.
Because complex line bundles are topologically classified by the first Chern class and principal $U(1)$-bundles by the Euler class in $H^2(B, \mathbb{Z})$, assume without loss of generality that
$$E = \big( [0, \infty) \times P \big) / \sim$$
where $p: P \rightarrow B$ is the principal $U(1)$-bundle over $B$ with Euler class $e(P)$ equal to $c_1(E)$ in $H^2(B, \mathbb{Z})$ and $`` \sim "$ denotes the the equivalence relation that collapses the $U(1)$ fibers to points in $\{ 0 \} \times P$.
If $\pi_2 : [0, \infty) \times P \rightarrow P$ denotes projection onto $P$, then $p \circ \pi_2 :  [0, \infty) \times P \rightarrow B$ induces the surjection $E \rightarrow B$.
Equivalently, $E$ is obtained from $P$ as the associated line bundle with fiber $\mathbb{C}$ with respect to the usual $U(1)$ action on $\mathbb{C}$.

For $a, b \in \mathbb{R}$, let $\hat{g}(a,b) = a^2 g_{U(1)} + b^2 p^*\check{g}$ denote the unique metric on $P$ such that $$p: \big( P, \hat{g}(a,b) \big) \rightarrow (B, b^2\check{g})$$ is a Riemannian submersion with homogeneous totally geodesic fibers of length $2 \pi a$ and whose horizontal distribution $(\ker p_*)^{\perp}$ equals that of the principal $U(1)$-connection on $P$ with curvature $q(d+2)\check{\omega} \in 2 \pi c_1(E)$.
Consider $U(1)$-invariant metrics on $E$ of the form
$$G(f,g) \doteqdot ds^2 + \hat{g}(f(s), g(s))  \qquad \text{ on } E \setminus B_0 = (0, \infty) \times P$$ where $s$ parametrizes $(0, \infty)$, $f,g : (0, \infty) \rightarrow \mathbb{R}_{>0}$, and $B_0 \subset E$ denotes the image of the zero section $B \hookrightarrow E$.
Observe that we make no assumptions on the isometry group of the base $(B, \check{g})$.
If $f,g$ satisfy suitable limiting conditions at $s=0$, then $G(f,g)$ extends to a smooth complete metric on $E$.
We investigate the existence and uniqueness of gradient steady soliton metrics of the form $G(f,g)$ on $E$, that is, metrics $G(f,g)$ satisfying $Rc = \nabla^2 h$ for some function $h(s)$.
In this case, the gradient steady Ricci soliton equation is equivalent to the following ODE system (cf. ~\cite{DancerWang08} equations (4.2-4.4)): 

\begin{equation} \label{solitonEqns}
\left\{ \begin{array}{ccl}
h_{ss} & = & -\frac{f_{ss}}{f} - d \frac{g_{ss}}{g} \\
\frac{f_{ss}}{f} & = & -d\frac{f_s g_s}{fg} - \frac{f_s h_s}{f} + A_3 \frac{f^2}{g^4}\\
\frac{g_{ss}}{g} & = & -\frac{f_s g_s}{fg} - \frac{g_s h_s}{g} - (d-1) \left( \frac{g_s}{g} \right)^2 + \frac{A_2}{d}\frac{1}{g^2} - 2\frac{A_3}{d}\frac{f^2}{g^4}\\ 
\end{array} \right.
\end{equation}
Here, $A_2, A_3$ are constants depending on the Einstein constant and the norm of the O'Neil tensor $||\mathcal{A}||$ for the Riemannian submersion $P \rightarrow B$.
Explicitly
\begin{equation*}
\begin{aligned}
A_2 &= d (d + 2)\\
A_3 &= d ||\mathcal{A} ||^2 = \frac{1}{4} d (d+2)^2 q^2 
\end{aligned}
\end{equation*}

In order to smoothly close up the metric at $s=0$, it is necessary and sufficient that $f$ extend smoothly to an odd function of $s$ with $f_s(0) = 1$, $g$ to an even function with $g(0)>0$, and $h_s$ to an odd function.
For $f,g,h_s$ solving the system (\ref{solitonEqns}), these conditions are equivalent to the right-hand sides of (\ref{solitonEqns}) having finite limits as $s \searrow 0$ and the following asymptotic behavior of $f,g,h_s$: 
\begin{equation} \label{leftasymps}
\begin{aligned}
\lim_{s \searrow 0} f(s) & =  0 	&\qquad& \lim_{s \searrow 0} f_s(s) & =  1 \\
\lim_{s \searrow 0} g(s) & > 0 	&\qquad&\lim_{s \searrow 0} g_s(s) & =  0 \\
\lim_{s \searrow 0} h_s(s) & = 0 
\end{aligned} 
\end{equation}

\begin{remark} \label{IveyCase}
Note that if $q = 0$ \big(i.e. $c_1(E) = 0 \in H^2(B, \mathbb{R})$\big) then (\ref{solitonEqns}) is the same as the system (1) considered by Ivey in ~\cite{Ivey94} (with $k=1$ and $n=d$).
In fact, because smooth complex Fano varieties are simply connected, $c_1(E) = 0$ in $H^2(B, \mathbb{Z})$ and $E \rightarrow B$ is the trivial line bundle $\underline{\mathbb{C}}_{B}$.
The one-parameter family of solutions Ivey obtains in ~\cite{Ivey94} therefore yields a one-parameter family of smooth complete $U(1)$-invariant gradient steady Ricci solitons on the total space of $\underline{\mathbb{C}}_{B}$ and proves theorem \ref{mainthm} in the case that $c_1(E) = 0$.
\end{remark}

Following Ivey ~\cite{Ivey94}, we introduce the change of variables 
\begin{equation} \label{changeofvars}
\begin{aligned}
X & =  \frac{g_s}{gh_s + dg_s + \frac{f_sg}{f}} 	&\qquad&Y  =&  \frac{1}{gh_s + dg_s + \frac{f_sg}{f}} \\
Z & =  \frac{\frac{f_sg}{f}}{gh_s + dg_s + \frac{f_sg}{f}} &\qquad&W  =&  \frac{\frac{f}{g}}{gh_s + dg_s + \frac{f_sg}{f}} \\
\end{aligned}
\end{equation}
and a new independent variable $t$ such that $\frac{g}{gh_s + dg_s + \frac{f_sg}{f}} dt = ds$.
The system (\ref{solitonEqns}) then becomes
\begin{equation}
\label{nonlin1}
\left\{ \begin{array}{ccl}
X_t & = & X(dX^2 + Z^2 -1) + \frac{A_2}{d}Y^2 - 2\frac{A_3}{d}W^2 \\
Y_t & = & Y(dX^2 + Z^2 - X) \\
Z_t & = & Z(dX^2+Z^2 -1) + A_3W^2 \\
W_t & = & W(dX^2 + Z^2 - 2X + Z )\\
\end{array} \right.
\end{equation}

\noindent Observe that the signs of $Y,W$ are constant for solutions $(X,Y,Z,W)$ of (\ref{nonlin1}) and that positivity of $Z$ is also preserved.
The asymptotics (\ref{leftasymps}) of $f, g,$ and $h_s$ at $s= 0$ imply that solutions $(X,Y,Z,W)$ of (\ref{nonlin1}) defined on $(t_{min}, t_{max})$ which correspond to soliton metrics must satisfy
$$\lim_{t \rightarrow t_{min}} (X,Y,Z,W)(t) = (0,0,1,0).$$
Such solutions necessarily have $t_{min} = -\infty$ since $(0,0,1,0)$ is a stationary solution of the ODE system (\ref{nonlin1}).
A summary of how to recover $f,g,$ and $h$ from solutions $(X,Y,Z,W)$ is given in Remark \ref{recover}.
We first describe the asymptotic behavior of solutions $(X,Y,Z,W)$ and a certain associated function $\cL$.

\begin{definition}
Given any solution $(X,Y,Z,W)(t)$ of the nonlinear ODE system (\ref{nonlin1}), we can recover $g$ (as a function of $t$) up to a multiplicative constant by solving $\frac{dg}{dt} = g X$.
Additionally, we may define $$\cL(t) \doteqdot g(t)Y(t)$$ up to this same multiplicative constant.
\end{definition}

\begin{remark}
$\cL^2$ is a constant multiple of the Lyapunov function considered in ~\cite{DancerWang09}.
\end{remark}

Denoting the scalar curvature by $S = \Delta_G h$, there is a first integral equation ~\cite{Ivey94} which states that $S + (h_s)^2 \equiv \mathcal{C}$ is constant.
In terms of $X,Y,Z,W$, the first integral equation says
\begin{equation}
\label{integraleqn}
dX^2+A_2Y^2+Z^2-A_3W^2 = 1 - \mathcal{C}\cL^2
\end{equation}
It is straightforward to verify using the ODE system that any solution $(X,Y,Z,W)(t)$ satisfies the first integral equation (\ref{integraleqn}) for appropriate choice of the constant $\cC$, which depends on the multiplicative constant in the definition of $g$.
Notice that if $\lim_{s \searrow 0} h_s(s) = 0$ then $\cC$ is the maximum value of the scalar curvature $S$ which is achieved on the image of the zero section $B_0 \subset E$.
By a result of Chen ~\cite{Chen09}, complete steady solitons have nonnegative scalar curvature. 
Hence, $\cC \ge 0$ with equality if and only if $S$ and $h_s$ are identically $0$.
Since we seek non-Einstein gradient steady solitons, we shall only consider solutions $(X,Y,Z,W)(t)$ which satisfy the first integral equation (\ref{integraleqn}) with $\cC>0$.
Note that conversely a smooth complete Ricci soliton $(G(f,g), h_s)$ on $E$ with $\cC > 0$ is not Einstein. 

\begin{prop} \label{lestimates}
Assume $(X,Y,Z,W)(t)$ is a solution of the nonlinear system (\ref{nonlin1}) defined on $(-\infty, t_{max})$ such that
\begin{enumerate}[(i)]
\item $Y,W> 0$, 
\item $\lim_{t\rightarrow - \infty} (X,Y,Z,W)(t) = (0,0,1,0)$, and
\item the constant $\cC$ in the first integral equation (\ref{integraleqn}) is positive $\cC>0$.
\end{enumerate}
Then the following asymptotics hold for $X,Y,Z,W$ as $t \rightarrow -\infty$:
\begin{enumerate}[(a)]
\item there exist constants $C_0' >0$ and $T_0\in (-\infty, t_{max})$ such that
$$ 0 \le X(t) \le C_0' e^{2t}$$
for all $t < T_0$,
\item for all $\epsilon > 0$, there exist constants $C_1, C_1' > 0$ and $T_1 \in (-\infty, t_{max})$ such that
$$C_1e^{(1-\epsilon)t} \le Y(t) \le C_1' e^{t} $$
for all $t < T_1$,
\item for all $\epsilon > 0$, there exist constants $C_2, C_2' > 0$ and $T_2 \in (-\infty, t_{max})$ such that
$$0 \le 1-Z(t) \le C_2' e^{2t} $$
for all $t < T_2$,
\item for all $\epsilon > 0$, there exist constants $C_3, C_3' > 0$ and $T_3 \in (-\infty, t_{max})$ such that
$$C_3e^{(2-\epsilon)t} \le W(t) \le C_3' e^{2t} $$
for all $t < T_3$, and
\item for all $\epsilon > 0$, there exist constants $C_4, C_4'> 0$ and $T_4 \in (-\infty, t_{max})$ such that
$$C_4 e^{(2-\epsilon)t} \le \cL^2(t) \le C_4'e^{2t}$$
for all $t < T_4$.
\end{enumerate}
Additionally, $\frac{X}{Y^2}$ remains bounded as $t \rightarrow -\infty$.
In particular, $|\int_{-\infty}^{T_0} X(t) dt| < \infty$ and $\lim_{t \rightarrow -\infty} \cL(t) = 0$.
\end{prop}

\begin{proof}
The proof proceeds by applying Gronwall's inequality and integrating the differential equations for $(X,Y,Z,W)$ in (\ref{nonlin1}).
Integrating $Y_t = Y(dX^2 + Z^2 - X)$ yields $$Y(t_1) = Y(t_0) e^{\int_{t_0}^{t_1} (dX^2 + Z^2 - X) dt}$$ for $t_0 < t_1$.
Since $\lim_{t \rightarrow -\infty} (X,Z) = (0,1)$ and $Y>0$, the integrand converges to 1 as $t \rightarrow -\infty$ and it follows that for every $\epsilon > 0$ there exist constants $C_1, C_1''>0$ and $T_1'$ such that $C_1e^{(1-\epsilon)t} \le Y(t) \le C_1'' e^{(1+ \epsilon)t} $ for all $t < T_1'$.

Integrating $W_t = W(dX^2 + Z^2 - 2X + Z )$ yields $$W(t_1) = W(t_0) e^{\int_{t_0}^{t_1} dX^2 + Z^2 - 2X + Z dt}$$ for all $t_0 < t_1$.
Since the integrand converges to 2 as $t \rightarrow -\infty$, it follows that for every $\epsilon > 0$ there exist constants $C_3, C_3''>0$ and $T_3'$ such that
$C_3e^{(2-\epsilon)t} \le W(t) \le C_3''e^{(2+\epsilon)t} $ for $t < T_3'$.

The first integral equation with $C>0$ implies that $$dX^2 + Z^2 - 1 = A_3W^2-A_2Y^2 - Cg^2Y^2 \le A_3W^2 - A_2Y^2.$$
It then follows from the asymptotics of $Y,W$ that $dX^2 + Z^2 - 1 \le A_3W^2-A_2Y^2 \le 0$ for $t\ll-1$.

The asymptotics of $Y,W$ also imply that $$X_t = X(dX^2+Z^2-1) + \frac{A_2}{d}Y^2 - 2\frac{A_3}{d}W^2 \ge X(dX^2+Z^2-1)$$ for $t\ll-1$.
Hence, $X(t_1) \ge X(t_0) e^{\int_{t_0}^{t_1} dX^2 + Z^2 - 1 dt}$ for $t_0 < t_1 \ll -1$ by Gronwall's inequality.
Since the integrand is bounded above by $0$ the exponential term is bounded between $0$ and $1$.
Thus, taking $t_0 \rightarrow -\infty$ yields $X(t) \ge 0$ for all $t \ll -1$.

Using the fact that $X \ge 0$ and $ dX^2 + Z^2-1 \le 0$ for $t \ll -1$ the above arguments may be strengthened to conclude that for every $\epsilon > 0$ there exist $C_1, C_1', C_3, C_3' > 0$ and $T_1, T_3$ such that 
$C_1e^{(1-\epsilon) t} \le Y(t) \le C_1'e^{t}$ for $t < T_1$ and
$C_3e^{(2-\epsilon)t} \le W(t) \le C_3'e^{2t}$ for $t < T_3$.

Because $0 \le X$ and $dX^2+Z^2-1 \le 0$ for $t \ll-1$, it follows that 
$$X_t = X(dX^2+Z^2-1)+\frac{A_2}{d}Y^2-2\frac{A_3}{d}W^2 \le \frac{A_2}{d}Y^2-2\frac{A_3}{d}W^2 \le \frac{A_2}{d} Y^2$$ for $t \ll-1$.
Hence, $X_t(t) \le C_1'e^{2t}$ for $t\ll-1$ and integrating with respect to $t$ yields that
$X(t) \le C_0'e^{2t}$ for all $t\ll-1$.

Moreover, for $t \ll -1$, $$X_t \le \frac{A_2}{d} Y^2 \implies X(t) \le \frac{A_2}{d} \int_{-\infty}^t Y(\tau)^2 d\tau < \infty$$
Similarly, for $t \ll -1 $, $$(Y^2)_t = 2Y^2(dX^2 + Z^2 - X) \ge Y^2 \implies Y(t)^2 \ge \int_{-\infty}^t Y(\tau)^2 d\tau$$
Hence,
$$\frac{X(t)}{Y(t)^2} \le \frac{A_2}{d} \text{ for $t \ll -1$,}$$
that is, $\frac{X}{Y^2}$ remains bounded as $t \rightarrow -\infty$.

Integrating $\frac{d}{dt} \cL^2 = 2\cL^2 (dX^2 +Z^2)$ yields that $$\cL^2(t_1) = \cL^2(t_0) e^{2\int_{t_0}^{t_1} dX^2 + Z^2 dt}$$ for all $t_0 < t_1$.
Because $dX^2+Z^2 \nearrow 1$ as $t \rightarrow -\infty$, it follows that for every $\epsilon > 0$ there exist constants $C_4, C_4' > 0$ and $T_4 \in \mathbb{R}$ such that $C_4 e^{(2 - \epsilon)t} \le \cL^2(t) \le C_4' e^{2t}$ for all $t < T_4$.

The fact that $0 \le 1-Z$ for all $t\ll -1$ follows from the fact that $\lim_{t \rightarrow -\infty} Z =1$ and $dX^2+Z^2-1 \le 0$ for all $t \ll -1$ .
For the upper bound on $1-Z$, note that for all $t \ll -1$
\begin{align*}
(1-Z)_t & = (1-Z) (dX^2 + Z^2 -1 ) -dX^2-Z^2 + 1 -A_3W^2\\
& =  (1-Z) (dX^2 + Z^2 -1 ) + A_2Y^2 - 2A_3W^2 +\cC\cL^2\\
& \le A_2Y^2 - 2A_3W^2 +\cC \cL^2\\
& \le A_2Y^2 + \cC\cL^2\\
& \le C_1'e^{2t} + C_4'e^{2t}
\end{align*}
The upper bound then follows from integrating with respect to $t$.
\end{proof}

\begin{definition}
Notice that $|\int_{-\infty}^{T_0} X(t) dt| < \infty$ implies $\lim_{t \rightarrow -\infty} g(t)$ exists and is finite.
We denote this limit by $$\lambda \doteqdot \lim_{t \searrow -\infty} g(t)$$ so $g(t) = \lambda e^{\int_{-\infty}^t X(\tau) d\tau}$.
Moreover, we are free to choose $\lambda \in \mathbb{R}$ since $g$ is only defined up to a multiplicative constant.
Of course, for $G(f,g)$ to be a metric, we must choose $\lambda > 0$.
\end{definition}

\comment{
\begin{remark}
Notice that the nonlinear ODE system has the symmetry $Y \mapsto -Y$.
Moreover, $Y_t = Y(dX^2 + Z^2 - X)$ implies that the sign of $Y$ is preserved along integral curves.
Henceforth, we can assume without loss of generality that our solution $(X,Y,Z,W)(t)$ of (\ref{nonlin1}) has $Y \ge 0$.
Note that $Y_t = Y(dX^2 + Z^2 - X)$ also implies that $Y > 0$ if and only if $Y(t_0) > 0$ for some $t_0$.

The same statement also holds for $W$.
Thus, the previous asymptotic estimates also hold if $Y$ or $W$ are negative with the sign of the respective constants $C_1, C_1', C_3, C_3'$ reversed.
\end{remark}
}

\begin{prop} \label{monotone}
Assume $(X,Y,Z,W)(t)$ is a solution of the nonlinear ODE system (\ref{nonlin1}) such that 
$\lim_{t \searrow -\infty} (X,Y,Z,W)(t) = (0,0,1,0)$ and
$| \int_{-\infty}^{T_0} X(t) dt| < \infty$ for some $T_0$ in the interval of existence for the solution.

If $Y(t_0) > 0$ for some $t_0$ and $\lambda > 0$, then $\cL(t)$ is strictly increasing.
Moreover, $\cL(t) > 0$ and $0 < \int_{-\infty}^{T_0} \cL(t) dt < \infty$.
\end{prop}

\begin{proof}
From the ODE system (\ref{nonlin1}), it follows that $\frac{d\cL}{dt} = \cL (dX^2 + Z^2)$.
If $Y(t_0) >0$ for some $t_0 \in \mathbb{R}$ then the ODE system (\ref{nonlin1}) implies that $Y > 0$ for all $t$.
Also, because $\lim_{t \rightarrow -\infty} Z(t) = 1$, $Z(t) > 0$ for all $t\ll -1$.
Since $$Z_t = Z(dX^2 + Z^2 -1 ) + A_3W^2 \ge Z(dX^2 + Z^2 -1),$$ it follows that $Z(t)>0$ for all $t$.
Hence, $\cL (dX^2 + Z^2) > 0$ and so $\cL(t)$ is strictly increasing.
The second statement then follows from the fact that $\lim_{t \rightarrow -\infty} \cL = 0$ and $\cL$ is strictly increasing.
\end{proof}

Recall that the constant $\cC$ in the first integral equation depends on $\lambda$.
From the asymptotics of solutions $(X,Y,Z,W)$, the following proposition recovers the exact dependence.

\begin{prop} \label{constants}
Let $(X,Y,Z,W)(t)$ be a solution of the ODE system (\ref{nonlin1}) such that
\begin{enumerate}[(i)]
\item $Y,W>0$, 
\item $\lim_{t \rightarrow -\infty} (X,Y,Z,W)(t) = (0,0,1,0)$, and
\item the constant $\cC$ in the first integral equation (\ref{integraleqn}) is positive $\cC > 0$ (a condition which is independent of the choice of $\lambda$).
\end{enumerate}
Then $$\cC \lambda^2 = 2 \left( \lim_{t \rightarrow -\infty} \frac{1-Z}{Y^2} \right) - A_2.$$
\end{prop}

\begin{proof}
It follows from the first integral equation (\ref{integraleqn}) that
$$\cC g^2 = \frac{-dX^2-A_2Y^2+A_3W^2-Z^2+1}{Y^2}$$
Taking the limit as $t \rightarrow -\infty$, we obtain
$$\cC \lambda^2 = \lim_{t \rightarrow -\infty} \frac{-dX^2+A_3W^2}{Y^2} + 2\lim_{t \rightarrow -\infty} \frac{1-Z}{Y^2} - A_2$$
The asymptotics of $X,W,Y$ at $-\infty$ imply that the first limit is $0$ and the result follows.
\end{proof}

\begin{remark} \label{recover}
Given a solution $(X,Y,Z,W)(t)$ of the nonlinear ODE system (\ref{nonlin1}) together with a distinguished point $t_0$ in the interval of existence, one can recover $f,g,h_s,$ and $s$ via
\begin{align*}
\frac{dg}{g} & = X dt &ds & =  gY dt = \cL dt\\
df & = g\frac{ZW}{Y} dt &h_s & = \frac{1- dX-Z}{\cL}\\
\text{Namely, }\qquad g(t) & = g(t_0) e^{\int_{t_0}^t X(\tau) d\tau} &s(t) & = s(t_0) + \int_{t_0}^t \cL(\tau) d\tau\\
f(t) & = f(t_0) + \int_{t_0}^t  g(\tau) \frac{Z(\tau)W(\tau)}{Y(\tau)} d\tau &h_s & = \frac{1- dX-Z}{\cL}
\end{align*}
Notice that we have some freedom to choose $g(t_0), f(t_0),$ and $s(t_0)$ as we see fit.
Additionally, in the setting of proposition \ref{monotone}, $\frac{ds}{dt} = \cL > 0$ implies $s(t)$ is strictly increasing as a function of $t$ and possesses an inverse. 
\end{remark}

\begin{thm} \label{recoverThm}
Assume $(X,Y,Z,W): (-\infty, t_{max}) \rightarrow \mathbb{R}^4$ is a solution of the nonlinear ODE system (\ref{nonlin1}) such that
\begin{enumerate}[(i)]
\item $Y, W, \cC, \lambda > 0$, and
\item $\lim_{t \rightarrow -\infty} (X,Y,Z,W)(t) = (0,0,1,0)$
\end{enumerate}
and $t_0 \in (-\infty, t_{max})$.
Let $f,g,h_s, s$ be obtained from $(X,Y,Z,W)(t)$ as specified above.
Then, for particular choices of $s(t_0)$ and $f(t_0)$,
\begin{align*}
\lim_{t \rightarrow -\infty} s(t) & = 0\\ 
\lim_{s \searrow 0} f(s) & =  0  &\lim_{s \searrow 0} f_s(s) & =   \lim_{t \rightarrow -\infty} \frac{W}{Y^2} \\
\lim_{s \searrow 0} g(s) & >  0 &\lim_{s \searrow 0} g_s(s) & =  0 \\
\lim_{s \searrow 0} h_s(s) & =  0 
\end{align*}
In particular, the metric 
$$G(f,g) = ds^2 + \hat{g} \big( f(s), g(s) \big) \qquad \text{ on } E \setminus B_0 = \big(0, s(t_{max})\big) \times P$$
extends smoothly over the zero section if and only if $$\lim_{t \rightarrow -\infty} \frac{W}{Y^2} = 1.$$
\end{thm}

\begin{proof} 
It follows from propositions \ref{lestimates} and \ref{monotone} that $0 < \int_{-\infty}^{t_0} \cL(\tau) d\tau < \infty$.
Taking $s(t_0) = \int_{-\infty}^{t_0} \cL(\tau) d\tau$ implies that $\lim_{t \rightarrow -\infty} s(t) = 0$.

Proposition \ref{lestimates} similarly implies that $\int_{-\infty}^{t_0} g(\tau) \frac{Z(\tau)W(\tau)}{Y(\tau)} d\tau$ is finite.
Setting $f(t_0) = \int_{-\infty}^{t_0} g(\tau) \frac{Z(\tau)W(\tau)}{Y(\tau)} d\tau$ thereby implies that $\lim_{s \searrow 0} f(s) = 0$.

The remaining limits follow from similar applications of proposition \ref{lestimates}, namely
\begin{equation*}
\begin{aligned}
\lim_{s \searrow 0} f_s(s) &= \lim_{t \rightarrow -\infty}  \frac{ZW}{Y^2} =  \lim_{t \rightarrow -\infty} \frac{W}{Y^2},\\
\lim_{s \searrow 0} g(s) &= \lim_{t \rightarrow -\infty} g(t) = \lambda > 0,\\
\lim_{s \searrow 0} g_s(s) &= \lim_{t \rightarrow -\infty}  \frac{X}{Y} = 0, \text{ and}\\
\lim_{s \searrow 0} h_s(s) &= \lim_{t \rightarrow -\infty} \frac{1-Z-dX}{\cL} = 0.
\end{aligned}
\end{equation*}

To prove the last statement, assume that $\lim_{t \rightarrow -\infty} \frac{W}{Y^2} = 0$.
Notice that it suffices to show that $\frac{g_s}{f}$ and $\frac{h_s}{f}$ have finite limits as $s \searrow 0$, for these are the only terms on the right-hand sides of (\ref{solitonEqns}) whose limits may not exist.
First, we claim that $\lim_{t \rightarrow -\infty} \frac{X}{Y^2} = \frac{A_2}{2d}$.
By proposition \ref{lestimates}, $\frac{X}{Y^2}$ is bounded as $t \rightarrow -\infty$.
Additionally,
$$\frac{X_t}{(Y^2)_t} = \frac{X}{Y^2} \frac{dX^2 + Z^2 - 1}{2(dX^2+Z^2-X)} + \frac{ \frac{A_2}{d} - 2 \frac{A_3}{d} \frac{W^2}{Y^2}}{2(dX^2+Z^2-X)}$$
The second term on the right-hand side limits to $\frac{A_2}{2d}$ as $t \rightarrow -\infty$ and the coefficient of $\frac{X}{Y^2}$ limits to $0$ as $t \rightarrow -\infty$.
Hence, a recursive form of l'Hopital's rule implies that $\lim_{t \rightarrow -\infty} \frac{X}{Y^2} = \frac{A_2}{2d}$.

It follows that
\begin{equation*}
\begin{aligned}
\lim_{s \searrow 0} \frac{g_s}{f} &= \lim_{t \rightarrow -\infty} \frac{ \frac{d}{dt} g_s}{\frac{df}{dt}}\\
&=\lim_{t \rightarrow -\infty} \frac{ ( \frac{X}{Y} )_t}{g \frac{ZW}{Y}}\\
&= \frac{1}{\lambda} \lim_{t \rightarrow -\infty} \frac{ -X + \frac{A_2}{d}Y^2 - 2 \frac{A_3}{d}W^2 + X^2}{W}\\
&=\frac{1}{\lambda} \lim_{t \rightarrow - \infty} \frac{ -\frac{X}{Y^2} + \frac{A_2}{d}}{\frac{W}{Y^2}} \qquad \text{(by proposition \ref{lestimates})}\\
&= \frac{A_2}{2d\lambda} \qquad \text{since } \lim_{t \rightarrow -\infty} \frac{X}{Y^2} = \frac{A_2}{2d} \text{ and } \lim_{t \rightarrow - \infty} \frac{W}{Y^2} = 1
\end{aligned}
\end{equation*}

Next,
\begin{equation*}
\begin{aligned}
\lim_{s \searrow 0} \frac{h_s}{g_s} & = \lim_{t \rightarrow -\infty} \frac{1-dX-Z}{gX}\\
& = \frac{1}{\lambda} \lim_{t \rightarrow -\infty} \frac{ \frac{1-Z}{Y^2} - d \frac{X}{Y^2}}{\frac{X}{Y^2}}\\
& = \frac{1}{\lambda} \frac{ \frac{1}{2}(\cC \lambda^2 + A_2) - \frac{A_2}{2}}{\frac{A_2}{2d}} \qquad \text{(by proposition \ref{constants})}\\
& = \frac{d \cC \lambda}{A_2}
\end{aligned}
\end{equation*}
Therefore, $\frac{h_s}{f} = \frac{h_s}{g_s} \frac{g_s}{f}$ has a finite limit of $\frac{\cC}{\lambda}$ as $s \searrow 0$.
\end{proof}

\section{Existence}
In light of theorem \ref{recoverThm}, we proceed towards the proof of theorem \ref{mainthm} by showing that there exist solutions $(X,Y,Z,W)(t)$ satisfying
$\lim_{t \rightarrow -\infty} \frac{W}{Y^2} = 1$.
Moreover, for any $\Lambda >0$, there exists a unique such solution modulo translation in $t$ satisfying $$\Lambda = \cC \lambda^2 = 2 \left( \lim_{t \rightarrow -\infty} \frac{1-Z}{Y^2} \right) - A_2.$$

Consider the following ODE system:
\begin{equation}
\label{nonlin2}
 \left\{  \begin{array}{ccl} 
X_{\tl{t}} & = & -X(dX^2 + \tl{Z}^2-2\tl{Z}) - \frac{A_2}{d}\tl{Y} + 2\frac{A_3}{d}W^2  \\
\tl{Y}_{\tl{t}} & = & -2\tl{Y}(dX^2 - X + \tl{Z}^2-2\tl{Z}+1)  \\
\tl{Z}_{\tl{t}} & = & -\tl{Z}(dX^2+\tl{Z}^2 -3\tl{Z}+2) +dX^2 + A_3W^2  \\
W_{\tl{t}} & = & -W(dX^2 - 2X +\tl{Z}^2 - 3\tl{Z}+ 2 ) 
\end{array}  \right.
\end{equation}
which is obtained from ODE system (\ref{nonlin1}) via the change of variables
\begin{equation*}
\begin{array}{ccl}
\tl{Y} & = & Y^2\\
\tl{Z}  & = & 1-Z\\
\tl{t} & = & -t\\
\end{array}
\end{equation*}
This system (\ref{nonlin2}) has a stationary solution at the origin $(X,\tl{Y},\tl{Z},W) = (0,0,0,0)$ with linearization given by
\begin{equation}
\label{lin2}
\frac{du}{d\tl{t}} = 
\left( \begin{array}{cccc}
0 & -\frac{A_2}{d} & 0 & 0 \\
0 & -2 & 0 & 0 \\
0 & 0 & -2 & 0 \\
0 & 0 & 0 & -2 \\ \end{array} \right) 
u.
\end{equation}
Note that the origin is not a hyperbolic equilibrium point.
Nonetheless, we should expect that there exist solutions in the stable manifold of the nonlinear system (\ref{nonlin2}) at the origin which quantitatively have the same asymptotic behavior as solutions of the linear system (\ref{lin2}) at the origin.
The goal of this section is to make this intuition rigorous and prove that
\begin{thm} \label{mainthmasymp}
Given $\alpha, \beta \in \mathbb{R}$, there exists a solution $(X, \tl{Y}, \tl{Z}, W)(\tl{t})$ of the nonlinear ODE system (\ref{nonlin2}) satisfying
\begin{enumerate}[(a)]
\item $\tl{Y}(\tl{t}) > 0$ for all $\tl{t}$,
\item $\lim_{\tl{t} \rightarrow \infty} e^{2\tl{t}}|(X, \tl{Y}, \tl{Z}, W)(\tl{t})| = 0$,
\item $\lim_{\tl{t} \rightarrow \infty} \frac{\tl{Z}}{\tl{Y}} = \alpha$, and
\item $ \lim_{\tl{t} \rightarrow \infty} \frac{W}{\tl{Y}} = \beta$.
\end{enumerate}
\end{thm}

Theorem \ref{mainthmasymp} ensures that we can recover $f, g, h_s$ with the proper limiting behavior at $s=0$ from the solution $(X, \tl{Y}, \tl{Z}, W)(\tl{t})$.

Note that we can clearly find a solution of linear ODE (\ref{lin2}) satisfying the properties of theorem \ref{mainthmasymp}.
In order to prove the theorem for solutions of the nonlinear system, we consider the nonlinear system (\ref{nonlin2}) as a perturbation of the linear system (\ref{lin2}) and use a fixed point argument to deduce that, given a suitable solution of the linear system (\ref{lin2}), there exists a solution of the nonlinear system (\ref{nonlin2}) which quantitatively has the same asymptotics as the given solution of the linear system (\ref{lin2}).
For simplicity, we'll omit the tildes from the notation (e.g. write $t$ instead of $\tl{t}$) in the remainder of this section as well as the next section.
Moreover, rewrite the linear system (\ref{lin2}) as
\begin{equation*}  \frac{du}{dt} =  Au \end{equation*}
and the nonlinear system (\ref{nonlin2}) as 
\begin{equation*} \qquad \frac{dv}{dt} = Av + b(v) \end{equation*}
where $A$ is the matrix in the linearization (\ref{lin2}) and $u,v$ are valued in $\mathbb{R}^4$.
Note that the components of $b(v) \in C^\infty (\mathbb{R}^4, \mathbb{R}^4)$ are degree 3 polynomials in $v_1, v_2, v_3, v_4$ with no constant or linear terms.

The linear system (\ref{lin2}) has fundamental solution matrix
\begin{align*}
\Phi(t,t_0) & = P
\left( \begin{array}{cccc}
1 & 0 & 0 & 0 \\
0 & e^{-2(t-t_0)} & 0 & 0 \\
0 & 0 & e^{-2(t-t_0)} & 0 \\
0 & 0 & 0 & e^{-2(t-t_0)} \\ \end{array} \right) P^{-1}
\\
& =  \left( \begin{array} {cccc}
1 & \frac{A_2}{2d}e^{-2(t-t_0)}-\frac{A_2}{2d} & 0 & 0 \\
0 & e^{-2(t-t_0)} & 0 & 0 \\
0 & 0 & e^{-2(t-t_0)} & 0 \\
0 & 0 & 0 & e^{-2(t-t_0)} \\
\end{array} \right) \\
\text{ where }
P &\doteqdot \left( \begin{array}{cccc}
1 & \frac{A_2}{2 d} & 0 & 0\\
0 & 1 & 0 & 0\\
0 & 0 & 1 & 0\\
0 & 0 & 0 & 1\\
\end{array} \right)
\end{align*}

If $u(t) = \Phi(t, 0)u_0$ is a solution of the linear system (\ref{lin2}) with initial condition $u_0 = P \cdot (0, y_0, z_0, w_0)$,
then $u(t) = (\frac{A_2}{2d} y_0 e^{-2t}, y_0 e^{-2t}, z_0 e^{-2t}, w_0e^{-2t})$ approaches the origin exponentially fast as $t \rightarrow +\infty$. 
Given such a solution $u$, we shall apply the Schauder fixed point theorem following Marlin \& Struble ~\cite{MarlinStruble69} to deduce the existence of a solution to the nonlinear ODE with quantitatively similar asymptotics to $u$ at $\infty$. 
In what follows, for $v \in \mathbb{R}^4$, $|v|$ will denote the norm of $v$ with respect to some fixed norm on $\mathbb{R}^4$.

\begin{prop}
The set $\mathbb{V} = \lbrace v \in C([0, \infty ) \rightarrow \mathbb{R}^4 ) |  \sup_{t \ge 0} e^{2t} |v(t)| < \infty \rbrace$ with the norm $||v|| = \sup_{t \ge 0} e^{2t} |v(t)| $ is a Banach space.
\end{prop}

\begin{proof}
It's clear that $(\mathbb{V}, \| \cdot \|)$ is a normed linear space.
To show completeness, note that the linear map given by $Sv(t) = e^{2t} v(t)$ defines an isomorphism of normed linear spaces from $(\mathbb{V}, \| \cdot \|)$ into the space of bounded continuous functions $C([0, \infty) \rightarrow \mathbb{R}^4)$ with the $\sup$ norm.
Since the latter space is complete, it follows that $(\mathbb{V}, \| \cdot \|)$ is complete.
\end{proof}

If $u(t) = \Phi(t, 0)u_0$ is a solution of the linear system (\ref{lin2}) with initial condition $u_0 = P \cdot (0, y_0, z_0, w_0)$,
then $u \in \mathbb{V}$.
So for any $r > 0$ we may consider the closed ball 
$$ \overline{B}_r(u) = \left\{ v \in C([0, +\infty) \rightarrow \mathbb{R}^4) \Big| \sup_{t \ge 0} e^{2t} |u(t)-v(t)| \le r \right\} \subset \mathbb{V}.$$
For such solutions $u(t)$ of the linear system (\ref{lin2}), we define the operator $$T_u: \overline{B}_r(u) \rightarrow C([0, \infty ) \rightarrow \mathbb{R}^4 )$$
$$T_u v(t) = u(t) - \int_t^\infty \Phi(t, s) b(v(s)) ds.$$
The following estimate will be used frequently throughout the remainder of this section and incidentally confirms that $T_u$ is well-defined.

\begin{prop} \label{estimate1}
There exists a constant $C$ (depending only on $u_0, r,$ and the polynomial components of $b$) such that, for $0 \le t_0 \le t$ and all $v \in \overline{B}_r(u)$,
$$e^{2t_0} \int_t^\infty |\Phi (t_0,s)| |b(v(s))| ds \le C e^{-2t}.$$
In particular, $e^{2t} \int_t^\infty |\Phi (t,s)| |b(v(s))| ds \le C e^{-2t}$ for all $t \ge 0$.

Moreover, if $0 < |u_0| < r$ and $r$ is sufficiently small, then the image $T_u \big( \overline{B}_r(u) \big)$ is contained in $\overline{B}_r(u) \subset \mathbb{V}$.
\end{prop}

\begin{proof}
For any $v \in \overline{B}_r(u)$, the triangle inequality implies $$|v(t)| \le |u(t)| + re^{-2t} \text{ for all } t \ge 0.$$
Since $u \rightarrow 0$ on the order of $e^{-2t}$, it follows that there exists a constant $C_0$ depending only on $u_0$ and $r$ such that $|v(t)| \le C_0e^{-2t}$ for all $t \ge 0, v \in \overline{B}_r(u)$.

Note that
\[ e^{2t} \Phi(t,s) = 
\left( \begin{array} {cccc}
e^{2t} & \frac{A_2}{2 d}e^{2s}-\frac{A_2}{2 d}e^{2t} & 0 & 0 \\
0 & e^{2s} & 0 & 0 \\
0 & 0 & e^{2s} & 0 \\
0 & 0 & 0 & e^{2s} \\
\end{array} \right) \]

It follows that for all $v \in \overline{B}_r(u)$ and $0 \le t_0 \le t$
\begin{align*}
\int_t^\infty |e^{2t_0} \Phi(t_0, s)| |b(v(s))| ds &  \le C_1  \int_{t}^\infty e^{2s} |b(v(s))| ds \\
& \le C_1 \int_{t}^\infty e^{2s} (K_1(C_0 e^{-2s})^2 + K_2(C_0 e^{-2s})^3) ds \\
& \le C_1K_1C_0^2 e^{-2t} + C_1 K_2 C_0^3 e^{-4t} \\
& \le C e^{-2t}
\end{align*}
where $K_1, K_2$ depend only on the polynomial components of $b$.
This completes the proof of the first statement of the proposition.
The second follows from taking $t_0 = t$.

Finally, if $|u_0| < r$ then we may take $C_0 = 2r$ above.
It then follows from the above estimates that for all $v \in \overline{B}_r(u)$ and $t \ge 0$
\begin{align*}
e^{2t} | T_u v(t) - u(t)| & =  \left| \int_t^\infty e^{2t} \Phi(t, s) f(v(s)) ds \right| \\
& \le C_1K_1C_0^2 e^{-2t} + C_1 K_2 C_0^3 e^{-4t} \\
& \le C_1 K_1 C_0^2 + C_1 K_2 C_0^3 \\
& = 4C_1 K_1 r^2 + 8C_1 K_2 r^3
\end{align*}
Therefore, if $0 < |u_0| < r$ is sufficiently small, then $$\sup_{t \ge 0} e^{2t} | T_u v(t) - u(t)| \le 4C_1 K_1 r^2 + 8C_1 K_2 r^3 \le r .$$
In fact, the above inequality can be made strict by taking $0 < |u_0| < r$ sufficiently small, and so the image can be taken to lie in the interior of the ball.
\end{proof}

\comment{
\begin{prop} \label{small}
If $r > |u_0|> 0$ is sufficiently small, then the image $T_u \big( \overline{B}_r(u) \big)$ is contained in $\overline{B}_r(u) \subset \mathbb{V}$.
\end{prop}

\begin{proof}
Note that the proof of the previous proposition \ref{estimate1} implies that if $|u_0| < d$ then we may take $C'=2d$.
It then follows from the proof of the previous proposition that  for all $v \in \overline{B}_r(u)$ and $t \ge 0$
\begin{align*}
e^{2t} | T_u v(t) - u(t)| & =  \left| \int_t^\infty e^{2t} \Phi(t, s) f(v(s)) ds \right| \\
& \le C_1K_1C'^2 e^{-2t} + C_1 K_2 C'^3 e^{-4t} \\
& \le C_1 K_1 C'^2 + C_1 K_2 C'^3 \\
& = 4C_1 K_1 r^2 + 8C_1 K_2 r^3
\end{align*}
Therefore, if $0 < |u_0| < r\ll1$ is sufficiently small, then $\sup_{t \ge 0} e^{2t} | T_u v(t) - u(t)| \le 4C_1 K_1 r^2 + 8C_1 K_2 r^3 \le r$.
\end{proof}
}

Now fix $r>0$ and a solution $u(t) = \Phi(t,0) u_0$ with $0 < |u_0|< r$ sufficiently small so that the third statement of the previous proposition \ref{estimate1} applies.
Henceforth, let $T \doteqdot T_u$ denote the nonlinear integral operator associated to this solution.
In order to deduce that the image $T\big(\overline{B}_r(u)\big)$ is precompact in $\mathbb{V}$, we require the following lemma:

\begin{lem} 
Assume $B \subset \mathbb{V}$ is a bounded set of continuous functions $[0, \infty) \rightarrow \mathbb{R}$ such that for every $\epsilon > 0$ there exists $t_1 \ge 0$ such that:
\begin{enumerate}
\item
\label{bound}
$t \ge t_1$ implies that $e^{2t} |v(t) - u(t)| \le \epsilon$ for all $v \in B$,
\item $B$ is a uniformly equicontinuous family on $[0, t_1]$.
\end{enumerate}
Then $B$ is precompact in $\mathbb{V}$.
\end{lem}

\begin{proof} 
Since $\mathbb{V}$ is complete, it suffices to show that $B$ is totally bounded (with respect to the norm $\| \cdot \|$ on $\mathbb{V}$).
Let $\epsilon > 0$.
By (\ref{bound}), choose $t_1 \ge 0 $ such that  $e^{2t} |v(t) - u(t)| \le \epsilon /4$ for all $v \in B$, $t \ge t_1$.
Let $B', \mathbb{V}'$ denote the images of the sets $B, \mathbb{V}$ respectively under restriction to the domain $[0, t_1]$.
Because $B'$ is bounded and equicontinuous, Arzela-Ascoli implies that we can cover $B'$ by a finite collection of balls $B_{\epsilon /4}(v_i) \subset \mathbb{V}'$.
Here these balls are defined with respect to the norm $$\| v(t) \|' \doteqdot \sup_{0 \le t \le t_1} e^{2t} |v(t)|$$ on $\mathbb{V}'$
which is equivalent to usual sup norm on $\mathbb{V}'$.
Assume without loss of generality that each of these balls $B_{\epsilon /4}(v_i)$ has nonempty intersection with $B'$.
It then follows that $e^{2t_1} | v_i(t_1) - u(t_1) | < \epsilon /2 $ for all $i$.
Extend the functions $v_i$ to the whole half line $[0, \infty)$ by setting
$v_i(s) = u(s) + e^{-2(s-t_1)}(v_i(t_1) - u(t_1))$ for all $s \ge t_1$.
Note that $u(t_1) + e^{-2(t_1 - t_1)}(v_i(t_1) - u(t_1)) = v_i(t_1)$ so the extension is well-defined and continuous on $[0, \infty)$.
Moreover, for all $s \ge t_1$
$$e^{2s} |v_i(s) - u(s)| = e^{2s}e^{-2(s-t_1)} |v_i(t_1) - u(t_1)|  = e^{2t_1} | v_i(t_1) - u(t_1) | < \epsilon / 2 .$$
Thus, the triangle inequality that the collection $\lbrace B_\epsilon (v_i) \rbrace$ covers $B$.
Hence, $B$ is totally bounded.
\end{proof}

\begin{prop}
The image $T\big(\overline{B}_r(u)\big)$ is precompact in $(\mathbb{V}, || \cdot ||)$.
\end{prop}

\begin{proof}
Let $\epsilon > 0$.
It follows from proposition \ref{estimate1} that, for all $v \in \overline{B}_r(u)$ and $t \ge t_1 \ge 0$,
\begin{align*}
e^{2t} | Tv(t) - u(t)| &= e^{2t} \left| \int_t^{\infty} \Phi(t,s) b(v(s)) ds \right| \\
& \le C e^{-2t}\\
& \le \epsilon \qquad \text{if } t_1 \text{ is sufficiently large.}
\end{align*}

For $v \in \overline{B}_r(u)$ and $0 \le t,t' \le t_1 \le t_2$,
\begin{align*}
|Tv(t) - Tv(t')| & \le |u(t) - u(t')| + \left| \int_t^\infty \Phi(t,s) b(v(s)) ds - \int_{t'}^\infty \Phi(t', s) b(v(s)) ds \right| \\
& \le  |u(t) - u(t')| + \left| \int_t^{t'} \Phi(t,s) b(v(s)) ds + \int_{t'}^\infty ( \Phi(t,s) - \Phi(t',s)) b(v(s)) ds \right| \\
& \le  |u(t) - u(t')| + \left| \int_t^{t'} \Phi(t,s) b(v(s)) ds \right| + \left| \int_{t'}^{t_2} ( \Phi(t,s) - \Phi(t',s)) b(v(s)) ds \right| \\ & \qquad+ \left| \int_{t_2}^\infty ( \Phi(t,s) - \Phi(t',s)) b(v(s)) ds \right| \\
& \le |u(t) - u(t')| + \int_t^{t'} |\Phi(t,s)|| b(v(s)) |ds + \int_0^{t_2} |\Phi(t,s) - \Phi(t',s) | | b(v(s))| ds \\ & \qquad+  \int_{t_1}^\infty | \Phi(t,s)| |b(v(s))| ds + \int_{t_2}^\infty | \Phi(t',s)| |b(v(s))| ds\\
\end{align*}
Now, for any $t_2 \ge t_1$, $u$ is uniformly continuous on $[0,t_2]$,
$\Phi$ is bounded and Lipschitz on $[0, t_2] \times [0, t_2]$, and
$\sup_{v \in \overline{B}_r(u)} \sup_{0 \le s \le t_2} b(v(s))$ is finite.
Let $\epsilon' > 0$.
By proposition \ref{estimate1}, the last two terms can be made less than $\frac{2}{5}\epsilon'$ \big(uniformly in $v \in \overline{B}_r(u)$\big) by taking $t_2$ sufficiently large.
For such a $t_2$, there exists $\delta = \delta(t_2) > 0$ such that $0\le t,t' \le t_1 \le t_2$ and $|t-t'| < \delta$ imply that the first three terms are less than $\frac{3 }{5}\epsilon'$ for all $v \in \overline{B}_r(u)$.
In other words, $T\big(\overline{B}_r(u)\big)$ is a uniformly equicontinuous family on $[0,t_1]$.
Hence, $T\big(\overline{B}_r(u)\big)$ satisfies the two conditions of the previous lemma and is therefore precompact in $\mathbb{V}$.
\end{proof}

\begin{prop}
$T : \overline{B}_r(u) \rightarrow \overline{B}_r(u)$ is continuous.
\end{prop}

\begin{proof}
Let $\epsilon > 0$ and $v_1, v_2 \in \overline{B}_r(u)$.
By the estimates in proposition \ref{estimate1} we can choose $t_1\ge 0$ sufficiently large so that, for all $v \in \overline{B}_r(u)$ and $ t > t_1$,
$$e^{2t}  \int_{t}^\infty |\Phi(t,s)| |b(v(s))| ds  < \epsilon / 3$$
and, moreover, that for all $v \in \overline{B}_r(u)$ and $0\le t \le t_1$,
$$e^{2t}  \int_{t_1}^\infty |\Phi(t,s)| |b(v(s)) | ds  < \epsilon / 3. $$
Hence, if $t > t_1$ then 
$e^{2t}|Tv_1(t) - Tv_2(t) | < 2 \epsilon / 3$.

Moreover, if $0 \le t \le t_1$ then
\begin{align*}
e^{2t}|Tv_1(t) - Tv_2(t) | & \le e^{2t} \left| \int_t^{t_1} \Phi(t,s) (b(v_1(s)) - b(v_2(s)) ) ds \right| + e^{2t} \left| \int_{t_1}^\infty \Phi(t,s) b(v_1(s)) ds \right| \\ & \qquad + e^{2t} \left| \int_{t_1}^\infty \Phi(t,s) b(v_2(s)) ds \right| \\
& \le e^{2t} \int_t^{t_1} | \Phi(t,s)| Lip(b) |v_1(s) - v_2(s)| ds + 2 \epsilon / 3 \\
& \le e^{2t} \int_t^{t_1} |\Phi(t,s)| Lip(b) e^{-2s} ds ||v_1 - v_2|| + 2 \epsilon / 3 \\
& < \epsilon
\end{align*}
if $||v_1 - v_2||$ is sufficiently small since $\int_t^{t_1} Lip(b) e^{2t-2s} | \Phi(t,s)|$ is bounded for $0 \le t \le s \le t_1$.
Here $Lip(b)$ denotes
$$Lip(b) \doteqdot \sup_{x,y \in N} \frac{|b(x) - b(y)|}{|x-y|}$$
where $N \subset \mathbb{R}^4$ is the bounded set $N \doteqdot \{ x \in \mathbb{R}^4 | dist(x, u([0, \infty)) \le r \}$.
Therefore, $T$ is continuous.
\end{proof}

The previous propositions indicate that we are now in a position to apply the Schauder fixed point theorem to prove theorem \ref{mainthmasymp}.

\begin{proof} (of theorem \ref{mainthmasymp})
First, choose $y_0, z_0, w_0 > 0$ so that $\frac{z_0}{y_0} = \alpha$ and $\frac{w_0}{y_0} = \beta$.
Next, take $0<r\ll1$ sufficiently small and rescale $y_0, z_0, w_0$ so that the above conditions still hold and additionally the assumptions of proposition \ref{estimate1} hold for the solution $u(t) = (\frac{A_2}{2d}y_0 e^{-2t}, y_0 e^{-2t}, z_0 e^{-2t}, w_0 e^{-2t})$ of the linear system (\ref{lin2}).
By the previous propositions, the Schauder fixed point theorem applies and so there exists a fixed point $v \in \overline{B}_r(u)$ of $T$.

Formally differentiating $v(t) = u(t) - \int_t^\infty \Phi(t,s) b(v(s)) ds$ implies that $v(t)$ is a solution of the nonlinear ODE system (\ref{nonlin2}).
To justify this claim rigorously, let $\psi_n(t) \doteqdot -\int_t^n \Phi(t,s) b(v(s)) ds$.
Then $\psi_n$ converges locally uniformly to $v - u$ because
\begin{align*}
\sup_{a \le t \le b} | \psi_n(t) - (v-u)(t)| & = \sup_{a \le t \le b} \left| \int_n^\infty \Phi(t,s) b(v(s)) ds \right| \\
& \le \sup_{a \le t \le b} \int_n^\infty |\Phi(t,s)|| b(v(s))| ds\\
& \le \sup_{a \le t \le b} Ce^{-2n}e^{-2t} \qquad \text{(if $n \ge b$ by proposition \ref{estimate1})}\\
& \le Ce^{-2n}e^{-2a}
\end{align*}
Therefore, $$\psi_n'(t) = b(v(t)) + A \psi_n(t) \text{ converges locally uniformly to }Av(t) - Au(t) + b(v(t)).$$
Hence, 
$$v'(t) = u'(t) + \lim_{n \rightarrow \infty} \psi'_n(t) = Av(t) + b(v(t)).$$

Since $||v - u|| < \infty$, $v(t)$ tends to $0$ as $t \rightarrow \infty$.
Moreover, it follows that
\begin{align*}
\lim_{t \rightarrow \infty} \frac{v_4}{v_2} 
& = \lim_{t \rightarrow \infty} \frac{u_4(t) - \int_t^\infty \sum_{j=1}^{4} \Phi_{4j}(t,s)b_j(v(s))ds}{u_2(t) - \int_t^\infty \sum_{j=1}^{4} \Phi_{2j}(t,s)b_j(v(s))ds}  \\
& = \lim_{t \rightarrow \infty} \frac{w_0 e^{-2t} - \int_t^\infty e^{-2t+2s} b_4(v(s)) ds}{y_0 e^{-2t} - \int_t^\infty e^{-2t+2s} b_2(v(s)) ds}\\
& = \lim_{t \rightarrow \infty} \frac{w_0 - \int_t^\infty e^{2s} b_4(v(s)) ds}{y_0- \int_t^\infty e^{2s} b_2(v(s)) ds}\\
& = \frac{w_0}{y_0}\\
& = \beta
\end{align*}
A similar argument shows that $\lim_{t \rightarrow \infty} \frac{v_3}{v_2} = \alpha$. 

It remains to check that $v_2 > 0$ for all $t$.
Since $v$ solves the nonlinear ODE system (\ref{nonlin2}), the sign of $v_2$ is constant for all $t$ and so it suffices to show that $v_2>0$ for some time $t$.
Because $$v_2(t) = y_0e^{-2t} - \int_t^\infty e^{-2t+2s} b_2(v(s)) ds = y_0e^{-2t} + o(e^{-2t})$$ and $y_0 > 0$, $v_2(t) > 0$ for $t$ sufficiently large.
\end{proof}

\begin{cor}
Given $\gamma \in \mathbb{R}$, there exists a solution $(X,Y,Z,W)(t)$ of the nonlinear ODE system (\ref{nonlin1}) defined on an interval containing $(-\infty, 0]$ such that
\begin{enumerate}
\item $Y(t) > 0$ for all $t$,
\item $\lim_{t \rightarrow -\infty} e^{-2t}|(X, Z, W)(t)-(0,1,0)| < \infty$,
\item $\lim_{t \rightarrow -\infty} e^{-t}|Y(t)| < \infty$,
\item $\lim_{t \rightarrow -\infty} \frac{W}{Y^2}  = 1$, and
\item $\lim_{t \rightarrow -\infty}  \frac{1-Z}{Y^2}  = \gamma$.
\end{enumerate}
\end{cor}

\begin{proof}
Apply theorem \ref{mainthmasymp} to obtain a solution $v = (v_1, v_2, v_3, v_4)$ of the nonlinear ODE system (\ref{nonlin2}) with $\alpha = \gamma$ and $\beta =1$.
Then obtain a solution $(X, Y, Z, W)(t)$ of the nonlinear ODE system (\ref{nonlin1}) given by the change of variables
\begin{align*}
X(t) & =  v_1(-t) \\
Y(t) & =  \sqrt{v_2(-t)} \\
Z(t) & =  1-v_3(-t) \\
W(t) & =  v_4(-t)
\end{align*}
It is straightforward to check that $(X,Y,Z,W)(t)$ satisfies the statement of the corollary.
\end{proof}

\begin{cor} \label{mainCor}
Given positive constants $\cC_0, \lambda_0 \in \mathbb{R}$, there exists a smooth gradient steady Ricci soliton $(G(f,g), h_s)$ on $E$ with $g(0) = \lambda_0$ and maximum scalar curvature $\cC_0$. 
\end{cor}

\begin{proof}
Let $\gamma \doteqdot \frac{1}{2} (C_0 \lambda_0^2 + A_2) > 0$ and $(X,Y,Z,W)(t)$ be a solution of the nonlinear ODE system (\ref{nonlin1}) with the properties of the previous corollary.
Let $t_0$ be in the domain of $(X,Y,Z,W)(t)$.
Define $f,g,s,h_s$ as in remark \ref{recover} with $g(t_0)$ chosen such that $\lambda = \lambda_0$. 
It follows from proposition \ref{constants} that $\cC = \cC_0$.
By theorem \ref{recoverThm}, $(G(f,g), h_s)$ is a smooth gradient steady Ricci soliton metric on $E$ with $g(0) = \lambda_0$ and maximum scalar curvature $\cC_0$.
\end{proof}

\section{Uniqueness} \label{uniquenessSection}
Next, we show that for given positive constants $\cC_0, \lambda_0$ the Ricci soliton in the previous corollary is unique.
The primary tool is the following result of Kellogg ~\cite{Kellogg76}.

\begin{thm} \label{Kellogg}
(Kellogg)
Let $B$ be a bounded convex open subset of a real Banach space $\mathbb{V},$ and let $T: \overline{B} \rightarrow \overline{B}$ be a compact continuous map which is continuously Fr\'echet differentiable on B.
Suppose that
\begin{enumerate}[(a)]
\item for each $v \in B$, 1 is not an eigenvalue of the derivative $DT_v$ of $T$ at $v$, and
\item for each $v \in \partial B, v \ne T(v).$
\end{enumerate}
Then $T$ has a unique fixed point.
\end{thm}

\noindent Thus, to show the uniqueness of the fixed point obtained in the previous section, it suffices to check that conditions $(a)$ and $(b)$ of theorem \ref{Kellogg} hold for the map $T$ from the previous section.

\begin{prop} \label{uniquefirst}
$T$ is continuously Fr\'echet differentiable on $B_r(u) = int(\overline{B}_r(u))$ with derivative 
$$DT_vh(t) = -\int_t^\infty \Phi(t,s) Db(v(s)) h(s) ds$$
where $Db$ denotes the derivative of $b : \mathbb{R}^4 \rightarrow \mathbb{R}^4$.
\end{prop}

\begin{proof}
\begin{align*}
||T(v+h) - Tv - DT_vh|| & =  \sup_t e^{2t} \left| \int_t^\infty \Phi(t,s) ( b(v(s)+h(s)) - b(v(s)) - Db(v(s))h(s) ) ds \right| \\
& \le  \sup_t \int_t^\infty e^{2t} |\Phi(t,s)| | b(v(s)+h(s)) - b(v(s)) - Db(v(s))h(s)| ds \\
& \le  \sup_t C_1 \int_t^\infty e^{2s}  |h(s)|^2 ds \qquad \text{(by the explicit form of $Db$ and $\Phi$)}\\ 
& \le C_1 ||h||^2 \sup_t \int_t^\infty e^{-2s} ds\\
& \le C_2 ||h||^2
\end{align*}
\end{proof}

\begin{prop}
For each $v \in int(\overline{B}_r(u))$, 1 is not an eigenvalue of $DT_v \in B(\mathbb{V}, \mathbb{V})$.
\end{prop}

\begin{proof}
Suppose there exists $h \in \mathbb{V}$ such that
$h(t) =-\int_t^\infty \Phi(t,s) Db(v(s)) h(s) ds$, or,
equivalently, $h$ is a fixed point of the operator $DT_v: \mathbb{V} \rightarrow \mathbb{V}$.
Let $a>0$.
If $\mathbb{V}_a$ denotes the image of $\mathbb{V}$ under restriction to the domain $[a,\infty)$ and if $h_a \in \mathbb{V}_a$ denotes $h$ restricted to $[a, \infty)$, then $h_a$ is a fixed point of $DT_v |_{\mathbb{V}_a}: \mathbb{V}_a \rightarrow \mathbb{V}_a$.
We claim that $DT_v |_{\mathbb{V}_a}: \mathbb{V}_a \rightarrow \mathbb{V}_a$ is a contraction for $a\gg1$ sufficiently large.
Indeed, for any $a >0$ and $u, w \in \mathbb{V}_a$,
\begin{align*}
||DT_v |_{\mathbb{V}_a} u - DT_v |_{\mathbb{V}_a} w||_{\mathbb{V}_a} & = \sup_{t\ge a} e^{2t} \left| \int_t^\infty \Phi(t,s) Db(v(s)) (u(s) - w(s)) ds \right| \\
& \le \sup_{t\ge a} \int_t^\infty e^{2t} |\Phi(t,s)| |Db(v(s))| |u(s) - w(s)| ds\\
& \le \| u - w \|_{\mathbb{V}_a} \sup_{t\ge a} \int_t^\infty C_1 e^{2s}e^{-2s}e^{-2s} ds\\
& \le C_1 \| u - w \|_{\mathbb{V}_a} \sup_{t \ge a} \int_t^\infty e^{-2s} ds\\
& \le \frac{1}{2}C_1 \| u - w \|_{\mathbb{V}_a} e^{-2a}
\end{align*}
where $C_1$ depends only on $\Phi$, the polynomial components of $b$, and $v: [0,\infty) \rightarrow \mathbb{R}^4$.
Because $C_1$ is independent of $a$ and $\| u-w \|_{\mathbb{V}_a}$ is nonincreasing in $a$, it follows that $DT_v |_{\mathbb{V}_a}: \mathbb{V}_a \rightarrow \mathbb{V}_a$ is a contraction for $a$ sufficiently large.
Therefore, the contraction mapping theorem implies that $DT_v |_{\mathbb{V}_a}$ has a unique fixed point in $\mathbb{V}_a$.
Since $h_a$ and $0$ are both fixed points of $DT_v |_{\mathbb{V}_a}$, it must be the case that $h_a = 0$ for $a\gg1$ sufficiently large.
Equivalently, $h(t) = 0$ for all $t\gg1$ sufficiently large.

Next, differentiating both sides of $h(t) =-\int_t^\infty \Phi(t,s) Df(v(s)) h(s) ds$ with respect to $t$ shows that $h: [0, \infty) \rightarrow \mathbb{R}^4$ satisfies the linear ODE system $$\frac{dh}{dt}(t) = \Big(A + Df\big(v(t)\big)\Big) h(t).$$
Uniqueness of solutions then implies that $h(t)=0$ for all $t\ge 0$.
Therefore, $1$ is not an eigenvalue of $DT_v$.
\end{proof}

\begin{prop} \label{uniquelast}
For $r$ sufficiently small, no $v \in \partial \overline{B}_r(u)$ is a fixed point of $T$.
\end{prop}

\begin{proof}
For $r$ sufficiently small, one can see from the proof of proposition \ref{estimate1} that the image of $\overline{B}_r(u)$ is contained in the interior of $\overline{B}_r(u)$.
\end{proof}

\noindent These propositions \ref{uniquefirst}-\ref{uniquelast} together indicate that theorem \ref{Kellogg} applies and so
\begin{cor}
The fixed point $v$ obtained in the proof of theorem \ref{mainthmasymp} is unique.
\end{cor}

\noindent Notice that the fixed point $v$ is only unique among $v$ in a small ball about the given solution $u$ of the linear system (\ref{lin2}) where $\|u\|$ is sufficiently small.
In fact, a stronger uniqueness result holds.

\begin{thm}
For any $u \in \mathbb{V}$, fixed points of the associated operator $T_u: \mathbb{V} \rightarrow \mathbb{V}$ are unique.
\end{thm}

\begin{proof}
Say $v, \tl{v} \in \mathbb{V}$ are fixed points of $T$.
For a given $a >0$, let $v_a \in \mathbb{V}$ denote the function $v_a(t) = v(t+a)$ and similarly define $\tl{v}_a, u_a$.
It follows that 
\begin{align*}
||v_a - u_a||_\mathbb{V} &\le \sup_{t \ge 0} e^{2t} | v_a(t) - u_a(t)|\\
&= \sup_{t \ge 0} e^{2t} |v(t+a) - u(t+a)|\\
&= \sup_{t \ge 0} e^{2t} \left| \int_{t+a}^\infty \Phi(t+a, s) f(v(s)) ds\right|  & \text{(change variables $\hat{t}=t+a$)}\\
&= \sup_{\hat{t} \ge a} e^{2(\hat{t}-a)} \left| \int_{\hat{t}}^\infty \Phi(\hat{t}, s) f(v(s)) ds \right| \\
&= e^{-2a} \sup_{t \ge a} e^{2t} \left| \int_{t}^\infty \Phi(t,s) f(v(s)) ds \right|\\
&\le e^{-2a} \sup_{t \ge 0} e^{2t} \left| \int_t^\infty \Phi(t,s) f(v(s)) ds \right|\\
&= e^{-2a} ||v - u||_\mathbb{V}.
\end{align*}
Moreover, $u \in \mathbb{V}$ implies that $||u_a||_\mathbb{V} \rightarrow 0$ as $a \rightarrow \infty$.
Therefore, there exists $a>0$ sufficiently large such that $v_a, \tl{v}_a$ lie in a suitably small ball around $u_a$ and $u_a$ has suitably small norm.
The above corollary then applies and so $v(t) = \tl{v}(t)$ for all $t \ge a$.
Differentiating $v = Tv$ and $\tl{v} = Tv$ with respect to $t$, it follows that $v$ and $\tl{v}$ solve the same ODE and satisfy $v(a) = \tl{v}(a)$.
Hence, $v(t) = \tl{v}(t)$ for all $t \ge 0$.
\end{proof}

\begin{thm} \label{uniqueness}
Let $E$ be the total space of a complex line bundle $E \rightarrow B$ over a Fano K\"ahler-Einstein base $B$ such that the first Chern class $c_1(E)$ of $E$ is a 
multiple of $c_1(B)$ in $H^2(B, \mathbb{R})$. 
If $(G(f, g), h_s)$ and $(G(\bar{f},\bar{g}), \bar{h}_s)$ are two $U(1)$-invariant gradient steady Ricci solitons on $E$ such that $g(0) = \bar{g}(0)$ and $\mathcal{C} = \bar{\mathcal{C}}$, then $(G(f, g), h_s) = (G(\bar{f},\bar{g}), \bar{h}_s)$.
\end{thm}

\begin{proof}
For each soliton $(G(f, g), h_s)$ and $(G(\bar{f},\bar{g}), \bar{h}_s)$, perform the associated change of variables to obtain solutions $v(\tl{t}), \bar{v}(\tl{t})$ of the nonlinear ODE system (\ref{nonlin2}).
We claim that $v, \bar{v}$ are fixed points of $T_u, T_{\bar{u}}$ respectively for possibly distinct solutions $u, \bar{u}$ of the linear ODE system (\ref{lin2}).
Note that for any $u$ solving the linear ODE system (\ref{lin2}), the asymptotics of $v$ imply that $T_u v$ is well-defined, that is the integral in the definition of $T_u$ converges.
Moreover, $T_u v$ satisfies the ODE system
$$\frac{d}{dt} T_u v = A (T_u v(t)) + f(v(t))$$
and $v$ similarly solves the ODE system
$$\frac{d}{dt} v = Av(t) + f(v(t)).$$
Now fix the solution $u$ of the linear system (\ref{lin2}) with $u(0) = v(0) + \int_0^\infty \Phi(0,s) f(v(s)) ds$ so that $T_u v(0) = v(0)$, and apply uniqueness of solutions to the ODE
$$\frac{d}{dt} w (t) = A w(t) + f(v(t))$$
to deduce that $T_u v = v$.
The same argument applies to show that $\bar{v}$ is a fixed point for the operator $T_{\bar{u}}$ associated to the appropriate $\bar{u}$.

Now, since $u$ solves the linear ODE system (\ref{lin2}), $u(t) = \Phi(t,0) u(0)$.
The explicit form of $u(0)$ can be partially determined from the asymptotics of $v$ as $\tl{t} \rightarrow +\infty$.
For example,
$$ 1= \lim_{\tl{t} \rightarrow +\infty} \frac{v_4}{v_2} = \lim_{\tl{t} \rightarrow + \infty} \frac{u_4}{u_2}= \frac{u_4(0)}{u_2(0)}$$
It follows that $u$ has the form
$$u(\tl{t}) = a \left( \frac{A_2}{2d}e^{-2\tl{t}}, e^{-2\tl{t}}, \gamma e^{-2\tl{t}}, e^{-2\tl{t}} \right)$$
and similarly
$$\bar{u}(\tl{t}) = \bar{a} \left( \frac{A_2}{2d}e^{-2\tl{t}}, e^{-2\tl{t}}, \bar{\gamma} e^{-2\tl{t}}, e^{-2\tl{t}} \right)$$
where $a, \bar{a}, \gamma, $ and $\bar{\gamma}$ are positive constants.
By remark \ref{recover}, we can exactly recover $(f,g,h_s)$ and $(\bar{f}, \bar{g}, \bar{h}_s)$ from $v$ and $\bar{v}$ by taking $\lambda = \bar{\lambda} = g(0)$.
It then follows from proposition \ref{constants} that in fact $\gamma = \bar{\gamma}$.
Hence, $u$ and $\bar{u}$ differ only by a translation in $\tl{t}$, i.e. $\bar{u}(\tl{t}) = u(\tl{t} + \tl{t}_0)$.
Uniqueness of fixed points of $T_{\bar{u}}$ then implies that $\bar{v}(\tl{t}) = v(\tl{t} + \tl{t}_0)$.
It then follows that the solitons $(G(f,g),h_s)$ and $(G(\bar{f}, \bar{g}), \bar{h}_s)$ recovered from $v$ and $\bar{v}$ by remark \ref{recover} are in fact identical.
\end{proof}

The above theorem shows not only that the Ricci soliton constructed in the proof of corollary \ref{mainCor} is unique given the choice of $\cC_0$ and $\lambda_0$ but also that any gradient steady Ricci soliton of the form $(G(f,g), h_s)$ on $E$ arises from that construction.
In particular, our construction recovers gradient steady Ricci solitons constructed by Ivey ~\cite{Ivey94}, Cao ~\cite{Cao96}, and Dancer-Wang ~\cite{DancerWang08}.

Recall that $\cC$ is the value of the maximum scalar curvature and $g(s)$ is the coefficient of the $\check{g}$-factor of the metric $G(f,g)$ on $E$.
Therefore, the theorem above states that $U(1)$-invariant gradient steady soliton metrics $(G(f,g), h_s)$ on $E$ are uniquely classified by their maximum scalar curvature and volume of the image of the zero section $B_0 \subset E$.


\begin{example}
When $B = \mathbb{CP}^{\frac{d}{2}}$ is complex projective space with the Fubini-Study metric, all complex line bundles over $\mathbb{CP}^{\frac{d}{2}}$ have first Chern class a multiple of $c_1(\mathbb{CP}^{\frac{d}{2}}) \in H^2(\mathbb{CP}^{\frac{d}{2}}, \mathbb{R})$.
Hence, theorem \ref{uniqueness} gives a classification result for smooth $U(1)$-invariant steady Ricci solitons on complex line bundles $E$ over complex projective space $\mathbb{CP}^{\frac{d}{2}}.$

When $E$ is the canonical line bundle, the uniqueness result shows that Cao's steady soliton ~\cite{Cao96} appears in this one-parameter family of Ricci solitons.
In fact, the one-parameter family on the canonical line bundle coincides with a collection of generalizations of Cao's steady soliton constructed by Dancer and Wang ~\cite{DancerWang08} (see remark \ref{shortproof} for more details).
In general, the soltions in this one-parameter family have at least $U(1)$ symmetry and Cao's steady soliton is the unique element with $U(\frac{d+2}{2})$ symmetry. 
\end{example}

\section{Completeness}
To finish the proof of theorem \ref{mainthm}, it remains to check that a suitable subset of the solitons constructed in corollary \ref{mainCor} are complete.
To simplify the exposition, we shall henceforth assume that $q \ne 0$ and appeal to remark \ref{IveyCase} for the case of $q = 0$.

\begin{lem} \label{lemmasign}
Assume $(X,Y,Z,W)(t)$ is a solution of the nonlinear ODE system (\ref{nonlin1}) defined on $(-\infty, t_{max})$ such that 
\begin{enumerate}[(i)]
\item $Y, W>0$,
\item $\lim_{t \searrow -\infty} (X,Y,Z,W)(t) = (0,0,1,0)$, and
\item $\cC>0$ in the first integral equation (\ref{integraleqn}).
\end{enumerate}
Then negativity of $X(t)$ and $(Z-X)(t)$ are forward invariant in $t$, that is,
$X(t_0) < 0$ for some $t_0 \in (-\infty, t_{max})$ implies $X(t) < 0$ for all $t > t_0$ and similarly for $Z-X$.

In particular, $\lim_{t \nearrow t_{max}} g(t)$ and $\lim_{t \nearrow t_{max}} \frac{W(t)}{Y(t)}$ exist (in $\mathbb{R} \cup \{ \pm \infty \}$).
\end{lem}

\begin{proof}
Suppose that $X$ is negative for some $t_0$.
Since $X \ge 0$ for $t \ll-1$, it follows that $t_* \doteqdot \inf \lbrace t \in \mathbb{R} | X(t) < 0 \rbrace$ is finite.
At $t_*$, $$X(t_*) = 0 \text{ and } X_t(t_*) \le 0 \implies \sqrt{\frac{A_2}{2A_3}} \le \frac{W}{Y}(t_*)$$
since $Y, W > 0$.
Now, negativity of $X$ persists (i.e. $X(t) < 0$ for all $t > t_*$).
Else, $t^* = \inf \{ t > t_* | X(t) \ge 0 \}$ exists and is finite.
At $t^*>t_0>t_*$,
$$X(t^*) = 0 \text{ and } X_t(t^*) \ge 0 \implies \frac{W}{Y}(t^*) \ge \sqrt{\frac{A_2}{2A_3}}.$$
Thus, the mean value theorem implies $\frac{d}{dt} \frac{W}{Y} \le 0$ for some $t \in (t_*, t^*)$, but $$\frac{d}{dt} \frac{W}{Y} = \frac{W}{Y} (Z - X)$$
is strictly positive for all $t \in (t_*, t^*)$.
Therefore, negativity of $X$ persists.

The proof for $Z-X$ is similar.
Namely, suppose that $(Z-X)(t_0) < 0$ for some $t_0$.
Then consider the non-empty open set $\{ t \in \mathbb{R} | (Z-X)(t) < 0 \}$.
Since $Z-X$ limits to $1$ as $t \rightarrow -\infty$, $t_* \doteqdot \inf \{t \in \mathbb{R} | (Z-X)(t) < 0 \}$ is finite.
At $t_*$,
$$(Z-X)(t_*) =0  \text{ and } (Z-X)_t (t_*) \le 0.$$
Thus, at $t_*$
$$(Z-X)_t(t_*) = A_3 W(t_*)^2 (1+\frac{2}{d}) - \frac{A_2}{d} Y(t_*)^2 \le 0$$
$$\implies \frac{W(t_*)^2}{Y(t_*)^2} \le \frac{A_2}{d}\left(A_3 + 2\frac{A_3}{d}\right)^{-1}.$$
Let $t^* \doteqdot \inf \{ t >t_* \in \mathbb{R} | (Z-X)(t) \ge 0 \}$
which a priori may be $+\infty$.
In fact, if $t^*$ is finite, then at $t^*$,
$$(Z-X)(t^*) =0  \text{ and } (Z-X)_t (t^*) \ge 0$$
$$\implies \frac{W(t^*)^2}{Y(t^*)^2} \ge \frac{A_2}{d}\left(A_3 + 2\frac{A_3}{d}\right)^{-1}.$$
Hence, the mean value theorem implies there exists $t_* < t < t^*$ such that $\frac{d}{dt} \frac{W^2}{Y^2} (t)$ is nonnegative.
However, $\frac{d}{dt} \frac{W^2}{Y^2} = 2 \frac{W^2}{Y^2}(Z-X) < 0$ for all $t_* < t < t^*$.
This contradiction indicates that negativity of $Z-X$ persists for all future time.

The final statement follows from observing that $\frac{dg}{dt} = g X$ and $\frac{d}{dt} \frac{W}{Y} = \frac{W}{Y} ( Z-X)$ implies that $g$ and $\frac{W}{Y}$ are monotonic for $t$ sufficiently large because $X$ and $Z-X$ are each either always nonnegative or negative for large $t$.
\end{proof}

\begin{prop}
Assume $(X,Y,Z,W)(t)$ is a solution of the nonlinear system (\ref{nonlin1}) defined on $(-\infty, t_{max})$ such that
\begin{enumerate}[(i)]
\item $Y,W> 0$, 
\item $\lim_{t\rightarrow - \infty} (X,Y,Z,W)(t) = (0,0,1,0)$, and
\item $\mathcal{C}>0$ in the first integral equation (\ref{integraleqn}).
\end{enumerate}
	Then $$(dX + Z)(t) \le 1 \qquad \forall t \in (-\infty, t_{max})$$
\end{prop}
\begin{proof}
	Observe that the first integral equation (\ref{integraleqn}) implies
	\begin{equation*} \begin{aligned}
		\frac{d}{dt} (dX + Z -1 ) =& (dX + Z-1)(dX^2 + Z^2 - 1) + dX^2 + Z^2 - 1+ A_2 Y^2 - A_3 W^2\\
		=& (dX + Z-1)(dX^2 + Z^2 - 1) - \mathcal{C} \cL^2\\
	\end{aligned} \end{equation*}
	Recall from the proof of proposition \ref{lestimates} that
		$$\lim_{t \to -\infty} dX + Z - 1 = 0 \qquad \text{ and } \qquad dX^2 + Z^2 - 1 \le 0 \quad \forall t \ll -1$$
	It then follows from the above differential equation
	 that $dX + Z \le 1$ for $t \ll -1$.
	The differential equation for $dX+Z - 1$ shows moreover that this condition is preserved for all $t$.
\end{proof}

Recall from section 3 that solutions $(X,Y,Z,W)(t)$ of (\ref{nonlin1}) satisfying $(i-iii)$ as above and $\lim_{t \to -\infty} \frac{W}{Y^2}(t) = 1$ are uniquely parametrized modulo translation in $t$ by
	$$ \mathcal{C} \lambda^2 = 2 \left( \lim_{t \to -\infty} \frac{ 1- Z}{Y^2} \right) - A_2$$

\begin{thm} \label{mainCthm}
	For all $d \ge 2$ and $q \ne 0$, there exists $\Lambda_0 = \Lambda_0(q,d) > 0$ such that
	if $(X,Y,Z,W)(t)$ is a solution of the nonlinear system (\ref{nonlin1}) defined on $(-\infty, t_{max})$ satisfying
\begin{enumerate}[(i)]
\item $Y,W> 0$, 
\item $\lim_{t\rightarrow - \infty} (X,Y,Z,W)(t) = (0,0,1,0)$, 
\item $\mathcal{C}>0$ in the first integral equation (\ref{integraleqn}),
\item $\lim_{t \to -\infty} \frac{W}{Y^2} = 1$, and
\item $\mathcal{C} \lambda^2 \ge \Lambda_0$ 
\end{enumerate}
	and $\lambda > 0$ then
		$$\frac{W^2}{Y^2}(t) \le \frac{A_2}{A_3 (d+2)} \qquad \forall  t \in (-\infty, t_{max})$$		
\end{thm}

This theorem completes the proof of theorem \ref{mainthm} as the following corollary shows.

\begin{cor} \label{mainCcor}
	Let $(X,Y,Z,W)(t)$ satisfy the assumptions of the previous theorem.
	If $\lambda >0$, then $t_{max} = +\infty$, $(X,Y,Z,W)(t) \to (0,0,0,0)$ as $t \to +\infty$, and $\mathcal{L} \to \frac{1}{\sqrt{ \mathcal{C}}}$ as $t \to +\infty$.
	
	In particular, the Ricci solitons constructed in corollary \ref{mainCor} with $g(0)= \lambda$ and maximum scalar curvature $\mathcal{C}$ are complete if $\mathcal{C} \lambda^2$ is sufficiently large.
\end{cor}
\begin{proof}
	First, observe that the differential equation for $X$ and the fact that $$\frac{W^2}{Y^2} \le \frac{A_2}{A_3(d+2)} < \frac{A_2}{2 A_3}$$ implies that $X \ge 0$ for all $t \in (-\infty, t_{max})$.
	Now,
	$$X \ge 0, Z \ge 0, \text{ and } dX+Z \le 1 \implies X,Z \text{ are bounded}$$
	Moreover,
		$$W^2 < \frac{A_2}{2 A_3} Y^2 \implies dX^2 + Z^2 + \frac{A_2}{2} Y^2 \le dX^2 + Z^2 + A_2 Y^2 - A_3 W^2 = 1 - \mathcal{C} \cL^2 \le 1$$
	and so $Y$ is also bounded.
	Finally, $\frac{W^2}{Y^2} \le \frac{A_2}{(d+2)A_3}$ implies $W$ is bounded.
	Therefore, $t_{max} = +\infty$.
	
	From proposition \ref{lemmasign}, $\cL, g,$ and $\frac{W}{Y}$ have limits (in $\mathbb{R} \cup \{ \pm \infty \}$) as $t \to +\infty$.
	The bound on $\frac{W^2}{Y^2}$ implies that the limits of $\cL$ and $\frac{W}{Y}$ are finite.
	Because the limits of $\cL, g$ and $\frac{W}{Y}$ as $t \rightarrow \infty$ exist, so do the limits of $Y$ and $W$ as $t \rightarrow \infty$.
The first integral equation (\ref{integraleqn}) then implies that $\lim_{t \rightarrow \infty} d X^2 + Z^2$ exists.
Hence, $\lim_{t \rightarrow \infty} \cL_t = \lim_{t \rightarrow \infty} \cL(dX^2 + Z^2)  = 0$, which implies that $$\lim_{t \rightarrow \infty} X = \lim_{t \rightarrow \infty} Z = 0.$$
Now,
\begin{equation*}
\begin{aligned}
\lim_{t \rightarrow \infty} Z_t = &\lim_{t \rightarrow \infty} Z(dX^2+Z^2 -1) + A_3W^2 \text{ exists}\\
\implies &\lim_{t \rightarrow \infty} Z(dX^2+Z^2 -1) + A_3W^2 = 0\\
\implies &\lim_{t \rightarrow \infty} W = 0\\
\end{aligned}
\end{equation*}
A similar argument applied to $X_t$ shows that $\lim_{t \rightarrow \infty} Y = 0$.
Finally, the claim about the asymptotic behavior of $\cL(t)$ follows immediately from taking the limit of the first integral equation (\ref{integraleqn}) as $t \rightarrow +\infty$.

Because the Ricci solitons constructed in corollary \ref{mainCor} satisfy the assumptions of the previous theorem when $\mathcal{C} \lambda^2 \ge \Lambda_0$ and $s(t) = \int_{-\infty}^t \cL(\tau) d\tau$, it immediately follows that the Ricci solitons constructed in corollary \ref{mainCor} are complete if $\mathcal{C} \lambda^2 \ge \Lambda_0$.
\end{proof}

Therefore, to complete the proof of theorem \ref{mainthm}, it remains to prove theorem \ref{mainCthm}.
\textit{For the remainder of this section, it will be assumed that $(X,Y,Z,W)(t)$ is a solution of (\ref{nonlin1}) satisfying assumptions $(i-iv)$ of theorem \ref{mainCthm} and $\lambda >0$.}
\textit{Additionally, suppose for contradiction that there exist $t \in (-\infty, t_{max})$ such that}
	$$\frac{W^2}{Y^2}(t) > \frac{A_2}{A_3(d+2)}$$
Define
	$$b \doteqdot \inf \left\{ t \in (-\infty, t_{max}) \left| \frac{W^2}{Y^2}(t) > \frac{A_2}{A_3(d+2)} \right. \right\} \in (-\infty, t_{max})$$
Note that $b$ may depend on $\mathcal{C} \lambda^2$.
To simplify the notation, consider the rescaled variables defined by
\begin{align*}
	\ol{X} =& \frac{X}{Y} & \ol{Y} =& \frac{1}{Y}\\
	\ol{Z} =& \frac{Z}{Y} & \ol{W} =& \frac{W}{Y}\\
	\ol{\cL} =& \frac{\cL}{Y} &  \ol{t}(t) =& \int_{-\infty}^t Y(\tau) d\tau
\end{align*}
Recall from proposition \ref{lestimates} that $Y(t)$ is integrable at $t =-\infty$ so $\ol{t}$ is well-defined.
Moreover, $Y > 0$ implies that $\ol{t}$ is an injective function of $t$ and so we may consider $\ol{X}, \ol{Y}, \ol{Z}, \ol{W}$ as functions of $\ol{t}$.
These rescaled variables satisfy the ODE system
	\begin{equation*} \begin{aligned}
		\frac{d}{d\ol{t}} \ol{X} &= \ol{X} \left( \ol{X}-\ol{Y} \right) + \frac{A_2}{d} - 2 \frac{A_3}{d} \ol{W}^2\\
		\frac{d}{d\ol{t}} \ol{Y} &= \ol{X} \ol{Y} - d\ol{X}^2 - \ol{Z}^2\\
		\frac{d}{d \ol{t}} \ol{Z} &= \ol{Z} \left( \ol{X}- \ol{Y} \right) + A_3 \ol{W}^2\\
		\frac{d}{d \ol{t}} \ol{W} &= \ol{W} \left( \ol{Z}-\ol{X} \right)\\
		\frac{d}{d \ol{t}} \ol{\cL} &= \ol{\cL}  \ol{X} \\
	\end{aligned} \end{equation*}
and in these variables the first integral equation (\ref{integraleqn}) reads
	$$d\ol{X}^2 + \ol{Z}^2 + A_2 - A_3 \ol{W}^2 = \ol{Y}^2 - \mathcal{C} \ol{\cL}^2$$
Moreover, $\ol{W}$ satisfies the second-order equation
		$$\frac{d^2}{d\ol{t}^2} \ol{W} = \left( \frac{d}{d \ol{t}} \ol{W} \right) \left( \ol{Z}-\ol{Y} \right) + \frac{A_3(d+2)}{d} \ol{W} \left(  \ol{W}^2 - \frac{A_2}{A_3(d+2)} \right)$$
Define $\ol{b} \doteqdot \ol{t}(b)>0$.
We collect some basic properties of $\left( \ol{X}, \ol{Y}, \ol{Z}, \ol{W} \right)\left( \ol{t} \right)$ on $\left(0, \ol{b} \right]$.
\begin{lem}
\begin{align}
 	&\lim_{\ol{t} \searrow 0 } \ol{W}  = 0 & \label{eq1} \\
	& \lim_{\ol{t} \searrow 0 } \ol{W}_{\ol{t}}  = 1  \label{eq2}\\
	&0 \le \ol{W}  < \sqrt{ \frac{A_2}{A_3(d+2)}} &\forall \ol{t} \in \left(0, \ol{b} \right) \label{eq3} \\
	&0 \le \ol{X} & \forall \ol{t} \in \left(0, \ol{b} \right] \label{eq4}\\
	&0 \le \ol{W}_{\ol{t}}  & \forall \ol{t} \in \left(0, \ol{b} \right] \label{eq5}\\
	&d \ol{X} + \ol{Z} \le \ol{Y} & \forall \ol{t} \in \left(0, \ol{b} \right] \label{eq6}\\
	&\ol{Z}- \ol{Y} \le 0 & \forall \ol{t} \in \left(0, \ol{b} \right] \label{eq7}\\
	& 0 \le \ol{\mathcal{L}}_{\ol{t}} & \forall \ol{t} \in \left(0, \ol{b} \right] \label{eq8}\\
	& \lim_{\ol{t} \searrow 0} \ol{\mathcal{L}} = \lambda & \label{eq9}
\end{align}
\end{lem}
\begin{proof}
	(\ref{eq1}) $$\lim_{\ol{t} \searrow 0 } \ol{W}( \ol{t}) = \lim_{t \searrow -\infty} \frac{W}{Y}(t) = 0$$
	(\ref{eq2}) $$\lim_{\ol{t} \searrow 0 } \ol{W}_{\ol{t}} ( \ol{t}) = \lim_{t \searrow - \infty} \frac{W}{Y^2} (Z-X)= 1$$
	(\ref{eq3}) $$\ol{W} = \frac{W}{Y} \implies \ol{W} > 0 \quad\forall \ol{t} \in (0 , \ol{b})$$ 
	The upper bound follows from the definition of $\ol{b}$.\\
	(\ref{eq4}) Follows from the fact that $X \ge 0 $ for $t \ll -1$ and the differential equation for $\ol{X}(\ol{t})$.\\
	(\ref{eq5}) From the definition of $\ol{b}$, $$0 \le \ol{W}_{\ol{t}} \left( \ol{b} \right) = \sqrt{ \frac{A_2}{A_3 (d+2)} } \left( \ol{Z} - \ol{X} \right) \left( \ol{b} \right) \implies (Z - X)(b) \ge 0$$
	It follows from lemma \ref{lemmasign} that $Z-X \ge 0$ for all $t \in (-\infty, b]$.
	Hence,
		$$\ol{W}_{\ol{t}} = \ol{W} \left( \ol{Z} - \ol{X} \right) \ge 0 \quad \forall \ol{t} \in \left(0, \ol{b} \right]$$
	(\ref{eq6}) This is the fact that $dX + Z \le 1$ restated in the rescaled variables.\\
	(\ref{eq7}) Follows from (\ref{eq6}) and (\ref{eq4}).\\
	(\ref{eq8}) Follows from (\ref{eq4}) and the fact that $Y, \mathcal{L} >0$.\\
	(\ref{eq9}) From the definition of $\cL$, $$\lim_{\ol{t} \searrow 0} \ol{\cL} = \lim_{t \searrow -\infty} \frac{ \cL}{Y} = \lim_{t \searrow -\infty} g(t) = \lambda$$
\end{proof}

\begin{lem} On $\left(0, \ol{b} \right]$,
	$$\frac{d}{d \ol{t}} \left( \ol{Z} - \ol{Y} \right) \le - \frac{1}{2} \mathcal{C} \lambda^2 + \frac{3}{2} \left( \ol{Z} - \ol{Y} \right)^2$$
	In particular, 
		$$\left( \ol{Z} - \ol{Y} \right) \left( \ol{t} \right) \le - \sqrt{ \frac{ \mathcal{C} \lambda^2}{3}}  \tanh \left( \frac{\sqrt{ 3 \mathcal{C} \lambda^2}}{2}  \ol{t} \right) \qquad \forall \ol{t} \in \left( 0, \ol{b} \right]$$
\end{lem}	
\begin{proof}
	For $\ol{t} \in \left(0, \ol{b} \right]$,
	\begin{equation*} \begin{aligned}
		\frac{d}{d \ol{t}} \left(\ol{Z} - \ol{Y} \right) =& \left( \ol{X} + \ol{Y} \right)\left(\ol{Z} - \ol{Y} \right) + \left(\ol{Z} - \ol{Y} \right)^2 + d \ol{X}^2 + A_3 \ol{W}^2\\
		\le& \ol{Y} \left(\ol{Z} - \ol{Y} \right) + \left(\ol{Z} - \ol{Y} \right)^2 + d \ol{X}^2 + A_3 \ol{W}^2 &&(\text{by \ref{eq4}})\\
		\le& \frac{1}{2} \left( \ol{Z} + \ol{Y}\right) \left( \ol{Z} - \ol{Y} \right) + \left(\ol{Z} - \ol{Y} \right)^2 + d \ol{X}^2 + A_3 \ol{W}^2 && (\text{by \ref{eq7}})\\
		=& \frac{1}{2} \left( - \mathcal{C} \ol{\cL}^2 - d \ol{X}^2 - A_2 + A_3 \ol{W}^2 \right) + \left(\ol{Z} - \ol{Y} \right)^2 + d \ol{X}^2 + A_3 \ol{W}^2 && (\text{by \ref{integraleqn}})\\
		=&- \frac{1}{2} \mathcal{C} \ol{\cL}^2 + \left(\ol{Z} - \ol{Y} \right)^2 + \frac{1}{2} d \ol{X}^2 - \frac{1}{2}A_2 + \frac{3}{2} A_3 \ol{W}^2\\
		\le& - \frac{1}{2} \mathcal{C} \lambda^2 - \frac{1}{2} A_2 + \left(\ol{Z} - \ol{Y} \right)^2 + \frac{1}{2} d \ol{X}^2 + \frac{3}{2} A_3 \ol{W}^2 && (\text{by \ref{eq8} \& \ref{eq9}})\\
		\le& -\frac{1}{2} \mathcal{C} \lambda^2 - \frac{1}{2}A_2 + \frac{3}{2} \left(\ol{Z} - \ol{Y} \right)^2 + \frac{3}{2} A_3 \ol{W}^2 && (\text{by \ref{eq4} \&  \ref{eq6}})\\
		\le& - \frac{1}{2} \mathcal{C} \lambda^2 + \frac{3}{2} \left( \ol{Z} - \ol{Y} \right)^2 && \left(\text{by \ref{eq3} \&  $d \ge 2$} \right)\\
	\end{aligned} \end{equation*}
	Recall that 
		$$\lim_{\ol{t} \searrow 0} \ol{Z} - \ol{Y} = \lim_{t \to -\infty} \frac{Z-1}{Y} = 0$$
	where the last equality follows from proposition \ref{constants}.
	The last statement of the proposition then follows from the comparison principle for ordinary differential equations.
\end{proof}

At this point, we have all the facts necessary to arrive at a contradiction and prove theorem \ref{mainCthm}.

\begin{proof}(of theorem \ref{mainCthm})
Recall that $\ol{W}$ satisfies
\begin{equation*} \begin{aligned}
	\frac{d^2}{d\ol{t}^2} \ol{W} =& \left( \frac{d}{dt} \ol{W} \right) \left( \ol{Z}-\ol{Y}\right) + \frac{A_3(d+2)}{d} \ol{W} \left( \ol{W}^2 - \frac{A_2}{A_3(d+2)} \right) \\
	\le&  \left( \frac{d}{d\ol{t}} \ol{W} \right) \left( \ol{Z}-\ol{Y} \right) && \left(\forall \ol{t} \in \left(0, \ol{b} \right] \right)\\
\end{aligned} \end{equation*}
 
We estimate the right-hand side using $\ol{W}_{\ol{t}} \ge 0$ and the bound
	$$\ol{Z} - \ol{Y} \le - \sqrt{ \frac{ \mathcal{C} \lambda^2}{3}} \tanh \left( \frac{\sqrt{ 3 \mathcal{C} \lambda^2}}{2}  \ol{t} \right)$$
from the previous proposition.
\begin{equation*} \begin{aligned}
	\frac{d^2}{d\ol{t}^2} \ol{W} \le& - \sqrt{ \frac{ \mathcal{C} \lambda^2}{3}}  \tanh \left( \frac{\sqrt{ 3 \mathcal{C} \lambda^2}}{2}  \ol{t} \right)\left( \frac{d}{d\ol{t}} \ol{W} \right) \\
	\implies \ol{W}_{\ol{t}} ( \ol{t}) \le&  \left[ \cosh \left( \frac{\sqrt{ 3 \mathcal{C} \lambda^2}}{2}  \ol{t} \right)  \right]^{-2/3} \\
	\implies \ol{W}(\ol{t}) \le&  \int_0^{\ol{t}} \left[ \cosh \left( \frac{\sqrt{ 3 \mathcal{C} \lambda^2}}{2}  \tau \right)  \right]^{-2/3} d \tau \\
	\le& \int_0^{\ol{t}} \left[ \frac{1}{2} e^{ \frac{\sqrt{ 3 \mathcal{C} \lambda^2}}{2}  \tau } \right]^{-2/3} d \tau\\
	\le& 2^{2/3} \int_0^{+\infty} e^{ - \sqrt{ \frac{ \mathcal{C} \lambda^2}{3} }  \tau } d \tau \\
	=&  2^{2/3} \sqrt{ \frac{3}{\mathcal{C} \lambda^2} } 
\end{aligned} \end{equation*}

Choose $\Lambda_0 > 0$ sufficiently large such that
	$$2^{2/3} \sqrt{ \frac{3}{\Lambda_0} } < \sqrt{ \frac{A_2}{A_3(d+2)} }$$
It follows that if $\mathcal{C} \lambda^2 \ge \Lambda_0$, then the above estimate implies that $\ol{W} \left( \ol{b} \right) < \sqrt{ \frac{A_2}{A_3 (d+2)}}$, a contradiction.
This contradiction indicates that for $\mathcal{C} \lambda^2 \ge \Lambda_0$, we have the bound
	$$\frac{W}{Y} (t) \le \sqrt{ \frac{A_2}{ A_3 (d+2)}} \qquad \forall t \in (-\infty, t_{max})$$
as claimed.
\end{proof}

\section{Geometric Properties}
In this the final section, we investigate certain geometric properties of these solitons.
\begin{thm}
The complete soliton metrics constructed in corollary \ref{mainCor} have nonnegative Ricci curvature.
\end{thm}

\begin{proof}
As $Rc = \nabla^2 h$, it suffices to show that the metric Hessian of $h$ is nonnegative or, equivalently, that $h_{ss}$, $\frac{f_s h_s}{f}$, and $\frac{g_s h_s}{g}$ are all nonnegative.
By proposition 2.3 of ~\cite{BDW15}, $h_s$ and $h_{ss}$ are positive. 
From the change of variables (\ref{changeofvars}), $f_s =  \frac{ZW}{Y^2}$ as functions of $t$.
Thus, positivity of $f_s$ follows from the fact that $Z$ and $W$ are positive for all $t \in \mathbb{R}$.

Finally, we claim that $g_s = \frac{X}{Y}$ is nonnegative.
Suppose for contradiction that $X$ is negative for some $t_*$.
Then the proof of lemma \ref{lemmasign} implies that $X(t)<0$ for all $t \ge t_*$.
It follows that $Y_t = Y(dX^2+Z^2-X) > 0$ for all $t \ge t_*$.
However, $Y(t)$ increasing on $[t_*, \infty)$ and $Y(t_*) > 0$ contradicts the fact from corollary \ref{mainCcor} that $Y \rightarrow 0$ as $t \rightarrow +\infty$.
This contradiction indicates that $X \ge 0$ for all $t \in \mathbb{R}$.
Therefore, $g_s \ge 0$ and $Rc \ge 0$.
\end{proof}

\begin{remark} 
A result of Bryant ~\cite{Bryant04} and Chau-Tam ~\cite{ChauTam05} states that a complete gradient K\"ahler-Ricci soliton with positive Ricci curvature which attains its maximum scalar curvature is necessarily biholomorphic to $\mathbb{C}^n$. 
The solitons constructed in corollary \ref{mainCor} achieve their maximum scalar curvature over the image of the zero section $B_0 \subset E$.
Moreover, theorem \ref{mainKahlerthm} shows that these solitons are K\"ahler when $c_1(E) = -c_1(B)$.
Hence, the result of Bryant and Chau-Tam indicates that these complete solitons will not have positive Ricci curvature in general.
\end{remark}

\begin{remark}
It is possible to show that, for any soliton $(G(f,g), h_s)$ on $E$ as in theorem \ref{mainthm},
\begin{equation*}
\begin{aligned}
&\lim_{s \rightarrow \infty} \frac{g(s)}{\sqrt{s}} \text{ is a finite positive number, and either }\\
(1) \qquad &\lim_{s \rightarrow \infty} \frac{f(s)}{\sqrt{s}} \text{ is a finite positive number, or }\\
(2) \qquad &\lim_{s \rightarrow \infty} f(s) \text{ is a finite positive number.}\\ 
\end{aligned}
\end{equation*}
In particular, the solitons have either paraboloid or cigar-paraboloid asymptotics.
Moreover, the asymptotics of $g(s)$ guarantee that in either case the scalar curvature satisfies the decay estimates in the asymptotically cylindrical hypothesis for Brendle's rigidity theorem on gradient steady Ricci solitons in dimensions greater than three ~\cite{Brendle14}.
\end{remark}

We now investigate when these soliton metrics on the total space are in fact K\"ahler metrics.
For a complex line bundle $E \rightarrow B$ over a Fano K\"ahler-Einstein base $(B, \check{g}, \check{J})$ such that $c_1(E) = q c_1(B)$ and a given smooth metric of the form $G(f,g)$, the total space $E$ admits a natural complex structure $J$ compatible with the metric $G(f,g)$.
Specifically, the complex structure $J$ takes the form $J = J_f \oplus \check{J}$ on the complement of the image of the zero section where $J_f$ is a complex structure on the fiber that depends only on the radial fiber coordinate $s$.
The details of this construction are contained in the appendix.

\comment{
Recall that we are in the setting of a complex line bundle $E \rightarrow B$ over a K\"ahler-Einstein base $B$ with positive Einstein constant such that 
$$c_1(E) = q c_1(B) \text{ in } H^2(B, \mathbb{R}).$$
Let $B_0 \subset E$ denote the image of the zero section $B \rightarrow E$.
If $G$ denotes a smooth complete Riemannian metric on the total space $E$ of the form 
$$G = ds^2 + f(s)^2 g_{2\pi} + g(s)^2 \check{g} \text{ on } E \setminus B_0,$$ 
then $E$ admits a natural complex structure compatible with $G$ of the form
$$J = \check{J} \oplus J_f$$
where $\check{J}$ denotes the complex structure on the base $B$ and $J_f$ is the natural complex structure on the fiber that depends only on the radial fiber coordinate.
The details of the construction of $J$ are contained in the appendix.}

It follows from the computations of the components of $\nabla J$ (contained in the appendix) that $(E, G(f,g), J)$ is K\"ahler if and only if 
\begin{equation}
-\frac{d + 2}{2} q f = g_s g.
\end{equation}

\noindent Under the assumption that $-\frac{d + 2}{2} q f = g_s g$, the soliton equations (\ref{solitonEqns}) simplify to
\begin{equation} \label{KahlerRicciSolitonEqns} 
\left\{
\begin{aligned}
2h_{ss} &= -d\frac{g_s h_s}{g} - h_s^2 + (d +2) \frac{h_s}{g g_s} \\
2 \frac{g_{ss}}{g} &= -(d+2) \frac{g_s^2}{g^2} - \frac{g_s}{g} h_s + \frac{d +2}{g^2} \\
\end{aligned} \right.
\end{equation}
Notice that, unlike the general Ricci soliton equations (\ref{solitonEqns}), $q$ does not appear in the system (\ref{KahlerRicciSolitonEqns}).

\begin{prop} \label{KahlerClosing}
If $(E, G, J)$ is a smooth complete K\"ahler-Ricci soliton where $G$ and $J$ are of the form above, then $c_1(E) = -c_1(B)$ and $E \rightarrow B$ is the canonical bundle of $B$.
\end{prop}

\begin{proof}
Differentiating the K\"ahler condition (\ref{KahlerRicciSolitonEqns}) with respect to $s$ implies that 
$$-\frac{d + 2}{2} q f_s = g_{ss} g + (g_s)^2,$$
which by the ODE system (\ref{KahlerRicciSolitonEqns}) becomes
$$-\frac{d + 2}{2} q f_s = -\frac{d+2}{2} g_s^2 - \frac{1}{2} g g_s h_s + \frac{d +2}{2} + (g_s)^2.$$
It follows from taking the limit as $s \rightarrow 0$ and $(\ref{leftasymps})$ that $q = -1$.
In other words, $c_1(E) = -c_1(B)$ in $H^2(B, \mathbb{R})$.

Because the Fano base $B$ is simply connected, it follows that $c_1(E) = -c_1(B)$ in $H^2(B, \mathbb{Z})$.
Therefore, $E \rightarrow B$ is the canonical bundle of $B$, as complex line bundles are classified by their first Chern class.
\end{proof}

\begin{thm} \label{mainKahlerthm}
For complex line bundles $E \rightarrow B$ with $c_1(E) = -c_1(B)$ in $H^2(B, \mathbb{R})$, the one-parameter family of Ricci solitons constructed in theorem \ref{mainthm} are K\"ahler with respect to the complex structure $J$.
\end{thm}

\begin{remark} \label{shortproof}
Let $E \rightarrow B$ be as in the statement of theorem \ref{mainKahlerthm}.
In ~\cite{DancerWang08}, the authors construct a one-parameter family of smooth complete steady gradient K\"ahler-Ricci solitons on $E$ (see Theorem 4.20(i) of ~\cite{DancerWang08} with $n_1=0$ and $r=2$).
Their solitons are of the form $(E, G(f,g), h_s)$ as in theorem \ref{mainthm} and the free parameters $\kappa_1 \le 0$ and $\sigma_2 > 0$ in their construction may be chosen to obtain any positive value of maximum scalar curvature $\cC$ and any positive value of $g(0)$.
Thus, theorem \ref{uniqueness} implies that our one-parameter family of Ricci solitons on $E$ is precisely the one-parameter family obtained in ~\cite{DancerWang08}.
In particular, all are complete regardless of the value of $\mathcal{C} \lambda^2$ and all are K\"ahler with respect to the complex structure $J$.
This argument proves theorem \ref{mainKahlerthm}, but, to remain self-contained, we include an independent proof of theorem \ref{mainKahlerthm}.
\end{remark}

\begin{proof}
After changing variables to $(X,Y,Z,W)(t)$ and then setting $\tl{Y}(t) = \frac{Y(t)^2}{X(t)}$, the ODE system (\ref{KahlerRicciSolitonEqns}) for a K\"ahler-Ricci soliton becomes the nonlinear homogeneous system
\begin{equation} \label{changeVarKahler}
\left\{ \begin{aligned}
X_t &= X\big(d X^2+(d+2)^2 \tl{Y}^2 - 2X - (d+2) \tl{Y} \big) \\
\tl{Y}_t &= \tl{Y} \big(d X^2 + (d+2)^2 \tl{Y}^2 - 3(d+2) \tl{Y} + 2\big) \end{aligned} \right.
\end{equation}
Note that solutions preserve the sign of $X$ and $\tl{Y}$.

Let $(X, \tl{Y})(t)$ be a solution of the initial value problem for the system (\ref{changeVarKahler}) with initial values $(X,\tl{Y})(t_0)$ chosen in the unstable manifold of $\left( 0, \frac{2}{d +2} \right)$ such that $X(t_0)>0$ and $\tl{Y}(t_0) - \frac{2}{d+2} < - \frac{A_2}{(d+2)^2} X(t_0)$.
There exists an open set in the plane of such initial values because
$(X, \tl{Y})(t) \equiv \left( 0, \frac{2}{d+2} \right)$ is a hyperbolic stationary solution with linearization given by
\begin{equation*}
\frac{d}{dt} u = 
\left( \begin{array}{cc}
2 & 0 \\
0 & 2\\
\end{array} \right)
u.
\end{equation*}

As $X, \tl{Y}$ remain positive for all $t$ such that the solution is defined, we can recover a solution of the ODE system ($\ref{nonlin1}$) by setting
\begin{equation*}
\begin{aligned}
X(t) &= X(t), & Y(t) &= \sqrt{ X(t) \tl{Y}(t)},\\
Z(t) &= (d+2) \tl{Y}(t) -1, & W(t) &= \frac{2}{d+2} X(t).
\end{aligned}
\end{equation*}
This solution $(X,Y,Z,W)(t)$ satisfies $Y, W >0$, $\lim_{t \rightarrow -\infty} (X,Y,Z,W)(t) = (0,0,1,0)$, and $\lim_{t \rightarrow -\infty} \frac{W}{Y^2} =1$.
Moreover, the condition that $$\tl{Y}(t_0) - \frac{2}{d+2} < - \frac{A_2}{(d+2)^2} X(t_0)$$ implies that the constant $\cC$ in the first integral equation is positive.
Indeed, in terms of $X, \tl{Y}$, the first integral equation is given by 
\begin{equation} \label{integralEqnK}
-\cC g^2 X = A_2 X + (d+2)^2 \tl{Y} - 2(d+2).
\end{equation}
$\tl{Y}(t_0) - \frac{2}{d+2} < - \frac{A_2}{(d+2)^2} X(t_0)$ thus implies that the right hand side of the integral equation is negative, and so $\cC$ is positive.
Therefore, by theorem $\ref{recoverThm}$, $(X,Y,Z,W)(t)$ yields a smooth complete soliton $(G(f,g),h_s)$ on $E$.

To confirm that this metric is indeed K\"ahler, we show that $W = \frac{2}{d+2}X$ implies $\frac{d+2}{2} f = g_s g$.
Indeed, $\frac{W}{Y} = \frac{f}{g} + \frac{C}{g}$ as functions of $t$ since both solve the scalar ODE
$$\frac{du}{dt} = -uX + \frac{ZW}{Y}.$$
Since proposition \ref{lestimates} implies that $\frac{W}{Y}$ and $\frac{f}{g}$ both limit to $0$ and $g$ limits to a positive constant as $t \rightarrow -\infty$, it follows that $C = 0$ and so $\frac{W}{Y} = \frac{f}{g}$.
Additionally, $\frac{X}{Y} = \frac{dg}{ds}$.
Thus, $W = \frac{2}{d+2}X$ implies that $\frac{d+2}{2} f = g_s g$ and hence the soliton metric is indeed K\"ahler.

In summary, we have shown that, from solutions $(X, \tl{Y})(t)$ of (\ref{changeVarKahler}) with initial values in a suitable open subset of the plane, we obtain K\"ahler-Ricci solitons $(G(f,g),h_s)$ on $E$.
To show that all such Ricci soliton metrics on $E$ in the one-parameter family are K\"ahler, it suffices by the uniqueness theorem \ref{uniqueness} to show that, given $\cC_0, \lambda_0 >0$, we can choose appropriate initial values $(X, \tl{Y})(t_0)$ such that the corresponding soliton $(G(f, g), h_s)$ has $g(s=0) = \lambda_0$ and
$$-\cC_0 g(t)^2 X(t) = A_2 X(t) + (d+2)^2 \tl{Y}(t) - 2(d+2).$$
Since $g(s)$ is only determined up to a multiplicative constant $\lambda$, we can always ensure $g(s=0) = \lambda_0$ by taking $\lambda = \lambda_0$.
Dividing the first integral equation (\ref{integralEqnK}) by $X$ and taking $t \rightarrow -\infty$ (or equivalently applying proposition \ref{constants}) implies that
$$-\cC \lambda^2 = (d +2)^2  \lim_{t \rightarrow -\infty} \frac{\tl{Y}-\frac{2}{d+2}}{X}  +A_2.$$
Thus, it suffices to show that we can choose initial conditions $(X, \tl{Y})(t_0)$ in the unstable manifold of $(0, \frac{2}{d+2})$ such that $X(t_0) > 0$ and 
$$-\cC_0 \lambda_0^2 = (d +2)^2  \lim_{t \rightarrow -\infty} \frac{\tl{Y}-\frac{2}{d+2}}{X}  +A_2.$$
Note that we expect that it is possible to choose such initial conditions since the linearization of the system (\ref{changeVarKahler}) at $(0, \frac{2}{d+2})$ is given by 
\begin{equation*}
\frac{d}{dt} u = 
\left( \begin{array}{cc}
2 & 0 \\
0 & 2\\
\end{array} \right)
u.
\end{equation*}
Indeed, to make this argument rigorous, we consider the nonlinear system (\ref{changeVarKahler}) as a perturbation of its linearization at $\left(0, \frac{2}{d+2} \right)$and apply the following theorem of Hallam and Heidel ~\cite{HH70}:

\begin{thm} \label{HH}
(Hallam-Heidel ~\cite{HH70}, theorem 2)
Consider the linear systems of differential equations
\begin{equation} \label{linHH}
\frac{du}{dt} = A(t) u
\end{equation}
with fundamental matrix $U(t)$ such that $U(t_0)$ is the identity and
\begin{equation} \label{nonlinHH}
\frac{dv}{dt} = A(t)v + f(t,v).
\end{equation}
Let $w(t,r) : [0, \infty) \times [0, \infty) \rightarrow [0, \infty)$ be continuous on its domain and nondecreasing in $r$ for $r>0$ and fixed $t \ge 0$.
Let $\Delta(t)$ be a nonsingular continuous matrix satisfying $$\| \Delta(t) U(t) \| \le \alpha(t)$$
where $\alpha(t)$ is a continuous positive function for $t \ge t_0 \ge 0$.
Assume that $f(t,x)$ satisfies
$$\| U^{-1}(t) f(t,x) \| \le w \left( t, \frac{ \| \Delta(t)v \|}{\alpha(t)} \right)$$
and that the scalar ODE $\frac{dr}{dt} = w(t,r)$ has a positive solution which is bounded on the interval $t \ge t_0$.
Then given any solution $u(t) = U(t) c$ of (\ref{linHH}) with $|c|$ sufficiently small, there exists a solution $v(t)$ of (\ref{nonlinHH}) such that 
$$\lim_{t \rightarrow \infty} \frac{\| \Delta(t) (v(t) -u(t)) \|}{\alpha(t)} = 0.$$
Here, $\| \cdot \|$ denotes any of the equivalent norms on these finite dimensional vector spaces.
\end{thm}

In this case, denote by $v(t) = \left( X(-t), \tl{Y}(-t)-\frac{2}{d+2} \right)$ solutions of
\begin{equation} \label{nonlin2HH}
\frac{d}{dt} v = 
\left( \begin{array}{cc}
-2 & 0 \\
0 & -2\\
\end{array} \right)
v + f(v),
\end{equation}
as a perturbation of its linearization
\begin{equation} \label{lin2HH}
\frac{d}{dt} u = 
\left( \begin{array}{cc}
-2 & 0 \\
0 & -2\\
\end{array} \right)
u
\end{equation}
which has fundamental solution matrix $U(t) = e^{-2t} Id$.
Here, $f :\mathbb{R}^2 \rightarrow \mathbb{R}^2$ has components which are polynomials of degree three with no constant or linear terms and is explicitly given by
\begin{equation*}
f(v) = 
\left( \begin{array}{c}
f_1(v_1, v_2) \\
f_2(v_1, v_2) \\
\end{array} \right) =
\left( \begin{array}{c}
-d v_1^3 - (d+2)^2 v_1 v_2^2 - 3(d+2)v_1 v_2 + 2 v_1^2\\
- d v_1^2 v_2 - (d+2)^2 v_2^3 - 3(d+2) v_2^2 -\frac{2d}{d+2} v_1^2 \\
\end{array} \right) 
\end{equation*}
Set $\Delta(t) = Id$ and $w(t, r) = D_1 e^{-2t}r^2$ where $D_1$ is a positive constant to be determined later, and notice that $\frac{dr}{dt} = w(t,r)$ has positive solutions which are bounded for $t \ge 0$.
Finally, set $\alpha(t) = \| Id \| e^{-2t}$ so that
$$\| \Delta(t) U(t) \| = \| U(t) \| \le \alpha(t) \qquad \text{ for all } t \in \mathbb{R}.$$
Because the components of $f$ consist of polynomials of degree three with no constant or linear terms, it follows that
$$\| U(t)^{-1} f(v) \| = e^{2t} |f(v)| \le D_2 e^{2t} \|v\|^2 = \frac{D_2 \| Id \|^2}{D_1} w \left( t, \frac{\|\Delta(t) v\|}{\alpha(t)} \right)$$
where $D_2 > 0$ is a constant depending on the choice of norms $\| \cdot \|$ and the polynomial entries of $f$.
Now set $D_1 = D_2 \| Id \|^2$ so that theorem \ref{HH} applies.

Let $c = (c_1, c_2)$ be a point in $\mathbb{R}^2$ such that $c_1 > 0$ and 
$$\frac{c_2}{c_1} = - \frac{C_0 \lambda_0^2 + A_2}{(d+2)^2} $$
By rescaling $c$, assume without loss of generality that $\| c \|$ is sufficiently small for the conclusion of theorem \ref{HH} to apply.
It then follows that there exists a solution $v(t)$ of (\ref{nonlin2HH}) such that
$$\lim_{t \rightarrow \infty} \frac{|\Delta(t) (x(t) -y(t))|}{\alpha(t)} = \lim_{t \rightarrow \infty} \frac{|v(t) - e^{-2t}c|}{e^{-2t}} = 0.$$
Recovering $X(t), \tl{Y}(t)$ from $v(t) = \left( X(-t), \tl{Y}(-t)-\frac{2}{d+2} \right)$, it follows that
\begin{equation*}
\begin{aligned}
\lim_{t \rightarrow -\infty} \frac{\tl{Y}(t) - \frac{2}{d +2}}{X(t)} &= \lim_{t \rightarrow \infty} \frac{v_2(t)}{v_1(t)}\\
&= \lim_{t \rightarrow \infty} \frac{v_2(t) - e^{-2t} c_2 + e^{-2t} c_2}{v_1(t) - e^{-2t}c_1 + e^{-2t}c_1}\\
&= \lim_{t \rightarrow \infty} \frac{ \left( \frac{v_2(t) - e^{-2t} c_2}{e^{-2t}} \right)+ c_2}{\left( \frac{v_1(t) - e^{-2t} c_1}{e^{-2t}}  \right)+ c_1}\\
&= \lim_{t \rightarrow \infty} \frac{c_2}{c_1}\\
&= \frac{c_2}{c_1}\\
&= - \frac{\cC_0 \lambda_0^2 + A_2}{(d+2)^2} 
\end{aligned}
\end{equation*}
Hence,
$$-\cC_0 \lambda_0^2 = (d +2)^2  \lim_{t \rightarrow -\infty} \frac{\tl{Y}-\frac{2}{d+2}}{X}  +A_2.$$
To complete the proof, it remains to check that $\lim_{t \rightarrow -\infty} (X, \tl{Y})(t) = \left( 0, \frac{2}{d+2} \right)$ and $X > 0$.
The former fact is clear since 
\begin{equation*}
\lim_{t \rightarrow \infty} \frac{|v(t) - e^{-2t}c|}{e^{-2t}} = 0 \implies \lim_{t \rightarrow \infty} v = 0 \implies \lim_{t \rightarrow -\infty} (X, \tl{Y})(t) = \left( 0, \frac{2}{d +2} \right).
\end{equation*}
Moreover,
$$\lim_{t \rightarrow \infty} \frac{|v(t) - e^{-2t}c|}{e^{-2t}} = 0 \implies \lim_{t \rightarrow \infty} e^{2t}X(-t) = c_1 >0.$$
Because the sign of $X$ is preserved, it must then be the case that $X>0$.
Thus, the smooth complete soliton $(G(f, g), h_s)$ recovered from $(X, \tl{Y})(t)$ has $g(s=0) = \lambda_0$ and constant $\cC = \cC_0$ in the first integral equation.
As $\cC_0, \lambda_0 > 0$ were given arbitrarily, the uniqueness theorem \ref{uniqueness} implies that every Ricci soliton $(G(f,g), h_s)$ in the one-parameter family constructed in theorem \ref{mainthm} can be obtained from a solution $(X, \tl{Y})(t)$ of (\ref{changeVarKahler}) in this way.
Hence, every such Ricci soliton $(G(f,g),h_s)$ on $E$ with $c_1(E) = -c_1(B)$ is in fact K\"ahler.
\end{proof}

\section{Appendix: The Complex Structure}
Recall the setting:
$(B^{d}, \check{g}, \check{J}, \check{\omega})$ is a K\"ahler-Einstein manifold with positive Einstein constant scaled such that $\check{Rc} = (d+2) \check{g}$.
Fix a suitable $q \in \mathbb{R}$ and let $P$ denote the principal $U(1)$-bundle with Euler class equal to $q c_1(B)$ in $H^2(B, \mathbb{R})$.
Given $a, b \in \mathbb{R}$, let $\hat{g}(a,b) = a^2 g_{U(1)} + b^2 \check{g}$ denote the unique 
metric on $P$ such that $$p: (P, \hat{g}(a,b)) \rightarrow (B, b^2\check{g})$$ is a Riemannian submersion with homogeneous totally geodesic fibers of length $2 \pi a$ and whose horizontal distribution $(\ker p_*)^{\perp}$ equals that of the principal $U(1)$-connection on $P$ with curvature $q(d+2) \check{\omega} \in 2 \pi c_1(E)$.

Throughout this section $\{ \hat{U} \} \cup \{ X_i \}_{i = 1}^{d}$ will denote a local orthonormal frame on $(P, \hat{g}(1,1))$ 
such that $-\hat{U}$ generates the family of diffeomorphisms given by the $U(1)$ action and the $X_i$ are basic vector fields for the Riemannian submersion $(P, \hat{g}(1,1)) \rightarrow (B, \check{g})$ such that $\{ \check{J} p_* X_i \}_{i = 1}^{d} = \{p_* X_i \}_{i = 1}^{d}$.

From $P$, form the associated complex line bundle $\tl{p} : E \rightarrow B$ for usual $U(1)$ representation on $\mathbb{C}$. 
$E$ is given topologically by $[0, \infty) \times P$ quotiented by the equivalence relation that collapses the circle fibers to points in $\{ 0 \} \times P$. 
In particular, the complement of the zero section $E \setminus B_0$ is diffeomorphic to $(0, \infty) \times P$.
We consider smooth $U(1)$-invariant smooth metrics on $E \setminus B_0 \cong (0, \infty) \times P$ of the form
$$G(f,g) = ds^2 + \hat{g}(f(s), g(s)) =ds^2 + f(s)^2 g_{U(1)} + g(s)^2 \check{g}$$
where $s$ is the coordinate on $(0, \infty)$.
Assume additionally that $f, g$ have suitable asymptotics as $s \rightarrow 0$ so that $G(f,g)$ extends to a smooth complete metric on $E$.
We shall abuse notation and also refer to $\hat{U}$ and $X_i$ as (possibly locally defined) vector fields on $E \setminus B_0$ via the identification $T(E \setminus B_0) \cong T((0, \infty) \times P) \cong \mathbb{R} \times TP$.

We now provide the details of the construction of the complex structure $J$ on the total space $E$.
The metric $G(f,g)$ determines a $G(f,g)$-orthogonal decomposition 
$$T_x E \cong \ker \tl{p}_*|_x \oplus (\ker \tl{p}_*|_x)^\perp $$
which for $x \in E \setminus B_0$ gives an isomorphism
$$T_x E \cong  Span_{\mathbb{R}} ( \partial_s, \hat{U}) \oplus T_{\tl{p}(x)} B .$$
The compatible complex structure is defined for $x \in E \setminus B_0$ by
$$J_x = (J_f)_x \oplus \check{J}_{\tl{p}(x)}$$
where $\check{J}$ denotes the complex structure on the base $B$ and $J_f$ denotes the endomorphism of $Span_{\mathbb{R}} ( \partial_s, \hat{U})$ defined by
$J_f \hat{U} = f(s) \partial_s$ and $J_f \partial_s = -\frac{1}{f(s)} \hat{U}$.
In particular, $J_f$ depends only on the coordinate $s \in (0, \infty)$.
It is straightforward to verify that $J$ defines a compatible almost complex structure on $(E \setminus B_0, G(f,g))$.

\begin{prop}
$J_f$ smoothly extends across the zero section $B_0 \subset E$ so that $J =  J_f \oplus \check{J}$ defines a compatible almost complex structure on $(E, G(f,g))$.
\end{prop}

\begin{proof}
The vector fields $\hat{U}$ and $\partial_s$ are defined on all of $E \setminus B_0$.
By the construction of $E$, a local trivialization on an open subset $U \subset B$ of the principal $U(1)$-bundle $P$ yields a trivialization $\psi : E|_U \rightarrow \mathbb{C} \times U $ such that $\psi_* \hat{U} = -\partial_{\theta}$ and $\psi_* \partial_s = \partial_s$ where $(s, \theta)$ denote polar coordinates of the $\mathbb{C}$ factor in the trivialization.
Since $J_f$ does not depend on $B$, it suffices work in the $\mathbb{C}$ factor only.
In these coordinates, $J_f$ takes the form
\begin{equation*}
\begin{aligned}
J_f \partial_{s} &= \frac{1}{f(s)} \partial_{\theta} = \frac{s}{f(s)} J_{std} \partial_s \\
J_f \partial_{\theta} &= -f(s) \partial_s = \frac{f(s)}{s} J_{std} \partial_{\theta}
\end{aligned}
\end{equation*}
where $J_{std}$ denotes the standard complex structure on $\mathbb{C}$.
Because $G(f,g)$ extends smoothly across the zero section, $\lim_{s \rightarrow 0} \frac{f(s)}{s}=1$.
Thus, $J_f$ extends continuously to $J_{std}$ on $T_0 \mathbb{C}$.
In fact, since $f$ is smooth, this extension is smooth.
Indeed, $f$ smooth and $f(0)=0$ implies that $\frac{f(s)}{s}$ is smooth on $[0, \infty)$.
Since $\lim_{s \rightarrow 0} \frac{f(s)}{s}=1$ and $f(s) > 0$ for $s>0$, it also follows that $\frac{s}{f(s)}$ is smooth on $[0, \infty)$ and so the smoothness of the extension of $J_f$ follows.

It remains to confirm that these extensions agree on overlaps of the trivializations.
Different trivializations of $P|_{U \cap V} $ with transition function $\Psi : U \cap V \rightarrow U(1)$ induce trivializations of $E|_{U \cap V}$ with the same transition function $\Psi : U \cap V \rightarrow U(1) \subset Aut(\mathbb{C})$.
Since any element of $U(1)$ preserves $J_{std}$ (as well as $J_f$), it follows that $J_{std} \in End( T_0 \mathbb{C} )$ is identified with itself under any $\Psi(x) \in U(1)$, $x \in U \cap V$.
Hence, the extensions agree on overlaps and so $J =J_f \oplus \check{J}$ extends across the zero section $B_0 \subset E$ to a compatible almost complex structure on $(E, G(f,g))$.
\end{proof}

\begin{prop}
J is an integrable almost complex structure on $E$.
\end{prop}

\begin{proof}
We check that the Nijenhuis tensor vanishes.
Since $\check{J}$ is an integrable complex structure on the base $B$ and $J_{f}$ doesn't depend on points in the base $B$, it suffices to show
$$N_{J_f} (\partial_s, \partial_{\theta}) = 0.$$
Indeed, this fact can be verified by computation:
\begin{equation*}
\begin{aligned}
N_{J_f} (\partial_s, \partial_{\theta}) &= [\partial_s, \partial_\theta] + J_f ( [J_f \partial_s, \partial_\theta] + [\partial_s, J_f \partial_\theta]) - [J_f \partial_s , J_f \partial_\theta]\\
&= J_f \left( \left[ \frac{1}{f} \partial_\theta, \partial_\theta \right] + \left[ \partial_s, -f \partial_s \right] \right) - \left[ \frac{1}{f} \partial_\theta, -f \partial_s \right] \\
&= J_f ( \partial_s (-f \partial_s) ) - f \partial_s ( \frac{1}{f} \partial_\theta ) \\
&= -\frac{f_s}{f} \partial_\theta + \frac{f_s}{f} \partial_\theta\\
&= 0.
\end{aligned}
\end{equation*}
\end{proof}

One may then compute the components of $\nabla J$ from the Riemannian submersion structures and the form of the metric $G(f,g)$, namely 
\begin{equation*}
\begin{aligned}
\nabla_{\partial_s} J \partial_s &= 0\\
\nabla_{\hat{U}} J \partial_s &= 0\\
\nabla_{X_i} J \partial_s &= \left(- \frac{d+2}{2}q \frac{f}{g^2} - \frac{g_s}{g} \right) \check{J} X_i\\
\nabla_{\partial_s} J \hat{U} &=0\\
\nabla_{\hat{U}} J \hat{U} &=0\\
\nabla_{X_i} J \hat{U} &=\left(  \frac{f g_s}{g} + \frac{d+2}{2} q \frac{f^2}{g^2} \right) X_i\\
\nabla_{\partial_s} J X_i &= 0\\
\nabla_{\hat{U}} J X_i &= 0 \\
\nabla_{X_k} J X_i &= \left( - g_s g - \frac{d+2}{2} q f \right) \delta_{kj} \partial_s  + \left( -\frac{g_s g}{f} - \frac{d+2}{2} q \right) \delta_{ki} \hat{U}
\end{aligned}
\end{equation*}
where $1 \le j \le d$ is the index such that $\check{J} p_* X_i = p_* X_j$. 

\bibliography{mybib}{}
\bibliographystyle{alpha}

\end{document}